\documentclass{article}

\usepackage{amsmath,amsfonts,amssymb,amsthm,mathrsfs}

\usepackage{xcolor}
\usepackage{changepage}
\usepackage{enumitem} 
\usepackage{graphicx}
\usepackage{tikz}
\usepackage[utf8]{inputenc}
\usepackage[T1]{fontenc}
\usepackage{indentfirst}
\usepackage{subfig}
\allowdisplaybreaks

\usepackage[todonotes={textsize=scriptsize}]{changes}

\numberwithin{equation}{section}

\newtheorem{myDefn}{Definition}[section]

\newtheorem{myProp}[myDefn]{Proposition}

\newtheorem{myExa}[myDefn]{Example}

\newtheorem{myTheorem}[myDefn]{Theorem}

\newtheorem{myRem}[myDefn]{Remark}

\def\HH{\mathrm{H}}
\def\LL{\mathrm{L}}
\def\R{\mathbb{R}}
\def\N{\mathbb{N}}

\def\nn{\mathrm{n}}

\newcommand{\fonction}[5]{\begin{array}[t]{lrcl}#1 :&#2 &\longrightarrow &#3\\&#4& \longmapsto &#5 \end{array}}

\newlist{primenumerate}{enumerate}{1}
\setlist[primenumerate,1]{label={\roman*$'$}}

\title{Analysis of a toy model for optimal crop protection}
\author{Luis Almeida\footnote{Sorbonne Université, Université Paris Cité, CNRS, Laboratoire Jacques-Louis Lions, LJLL, F-75005 Paris, France \texttt{luis.almeida@sorbonne-universite.fr}}, Aymeric Jacob de Cordemoy\footnote{Sorbonne Université, Université Paris Cité, CNRS, Laboratoire Jacques-Louis Lions, LJLL, F-75005 Paris, France \texttt{aymeric.jacob\_de\_cordemoy@sorbonne-universite.fr}}, Ayman Moussa\footnote{Sorbonne Université, Université Paris Cité, CNRS, Laboratoire Jacques-Louis Lions, LJLL, ENS-PSL, Département de Mathématiques et Applications (DMA), F-75005 Paris, France. \texttt{ayman.moussa@sorbonne-universite.fr}}, Nicolas Vauchelet\footnote{LAGA, CNRS UMR 7539, Institut Galilée, Université Sorbonne Paris Nord, 99 avenue Jean-Baptiste Clément, 93430 Villetaneuse, France. \texttt{vauchelet@math.univ-paris13.fr}}
}

\begin{document}

\maketitle

\begin{abstract}
In this paper we investigate an optimal control problem involving a toy model for the protection on a crop
field. Precisely, we consider a protection on a crop field and we want to place intervention zones represented by a control, in order to maximize the protection on the field during a given period. We prove that there exists a unique control which maximizes the protection and, moreover, it must be a bang-bang control. Furthermore, with additional assumptions on the crop field geometry, some results on the shape of the optimal intervention are proved using comparison results for elliptic equations via Schwarz and Steiner symmetrizations. Finally, some numerical simulations are performed in order to illustrate those results.
\end{abstract}

\textbf{Keywords:} Optimal control, bang-bang controls, Schwarz symmetrization, Steiner symmetrization, elliptic equations, population dynamics.

\textbf{AMS Classification:} 35Q92; 49J99; 49J30; 49K20; 49Q10.

\section{Introduction}

\paragraph{Motivation}

\textit{Population dynamics} is a branch of mathematical biology that studies the fluctuation over time of the number of individuals in a population of living beings. The encounter of optimal control techniques with population dynamics led to remarkable examples of applications in applied mathematics. For instance: optimal intervention strategies have been studied to control epidemics (see, e.g.,~\cite{BEHNCKE,GREENHALGH} for more general settings and~\cite{OptimCOVID} for an application to COVID19); in~\cite{OptimInsecticide} the use of insecticide is optimized in the fight against arboviruses; in~\cite{WOLB2,OptimWol,WOLB} the authors design release protocols in a population replacement strategy.
We refer also to~\cite{LOU2} where the author studies how migration and spatial heterogeneity of the environment affect the total population of a single and multiple species, to~\cite{YANA} that deals with an indefinite weight linear eigenvalue problem related with biological invasions of species, to~\cite{MAZPRINAD} where the authors study the problem of optimizing the total population size for a standard logistic-diffusive model is studied, to~\cite{MAZ} that investigates the optimization of the carrying capacity in logistic diffusive models, to~\cite{LOU} which focuses on how to maximize the total population of a single species with logistic
growth living in a patchy environment and, finally, to~\cite{EIJI} that studies the effect of spatial heterogeneity on the total population of
a biological species at a steady state, using a reaction–diffusion logistic model.

In this paper, we are interested in controlling crop diseases such as Cercospora leaf spot, beet yellows virus, powdery mildew, etc., which are one of the main causes of yield loss in vegetable crops. Phytosanitary products are classically used to fight their spread, but due to their negative impact on the environment it is important to reduce their use and to find alternative strategies. For instance, in~\cite{GIR} the authors develop an agro-ecological approach to control aphid populations in sugar beet fields using natural predators in order to prevent the spread of viruses transmitted by these pests. Inspired by this problematic and the references cited above, we consider in this paper an optimal control problem involving a toy model for the protection on a crop field.

\medskip

\paragraph{Description of the optimal control problem}
Let~$n\in\N^*$ be a positive integer and $\Omega$ be a nonempty bounded connected open subset of $\R^{n}$, with a lipschitz boundary $\partial{\Omega}$. Let $T>0$ be the final time and consider the following toy model for the protection on a crop field given by
\begin{equation}\label{modele}
\arraycolsep=2pt
\left\{
\begin{array}{rcll}
\partial_t \phi  & = &  -\alpha(1-u)\phi & \text{ in } \Omega\times]0,T[ , \\
\phi(\cdot,0) & = & 1 & \text{ in } \Omega,\\
  p_u - D\Delta p_u & = &  \phi(u) & \text{ in } \Omega\times]0,T[ , \\
p_u & = & 1  & \text{ on } \partial{\Omega}\times]0,T[,
\end{array}
\right.
\end{equation}
for some~$u\in\LL^{\infty}(\Omega)$  such that $0\leq u \leq 1$ \textit{a.e.} in $\Omega$  and~$D>0$ is the diffusion coefficient. This model has the following interpretation. The domain~$\Omega$ represents a crop field, the function~$\phi(u)$ represents the active amount of a protection inducing mechanism (for instance, the presence of some predator species) and~$p_u(x,t)$ the protection due to $\phi(u)$ at time~$t>0$ and at a position~$x\in \Omega$ in the field. The dynamics of the protection is assumed to be much faster than that of the disease, thus the protection in the field adapts to the control immediately and reaches an equilibrium. This protection is assumed to be perfect at initial time, i.e.~$p_u(x,0)=1$ \textit{a.e.} in $\Omega$, and degrades over time with a rate denoted~$\alpha>0$. In particular, note that the solution $\phi$ is given by $\phi(u):  \Omega\times]0,T[ \rightarrow \phi(u)(x,t)=\mathrm{e}^{-\alpha(1-u(x))t}\in\R^+$. We want to place some intervention zones (for instance a predator reservoir) in the domain such that the protection remains as high as possible throughout the season. Then, we introduce a control variable~$u\in\LL^{\infty}(\Omega)$ such that~$0\leq u \leq 1$ \textit{a.e.} in~$\Omega$, describing those interventions: in particular if~$u=1$ the intervention is maximal and if~$u=0$ there is no intervention. Obviously at the position where $u=1$ the protection does not degrade, whereas at the position~$u=0$ the degradation of the protection is maximal. Given an amount of possible total intervention, i.e.~$\int_\Omega u \leq L$, where~$L\in[0,|\Omega|]$, we focus on the question of knowing where and how to intervene in order to maximize the protection during a given period of time $[0,T]$. Thus this leads to the following optimal control problem
\begin{equation}\label{optimalcontrol}
            \max\limits_{ \substack{ u\in \mathcal{U}_{ad} \\  } } \; \mathcal{J}(u), 
\end{equation}
where 
$$
\mathcal{U}_{ad}:=\left\{ u\in\LL^{\infty}(\Omega) \mid 0\leq u \leq 1 \text{ \textit{a.e.} in } \Omega, \text{ and } \int_\Omega u (x)\mathrm{d}x \leq L \right\},
$$
and where $\mathcal{J}$ is the cost functional defined by
$$
\fonction{\mathcal{J}}{\LL^{\infty}(\Omega)}{\R}{u}{\mathcal{J}(u) := \displaystyle\int_0^T \int_\Omega p_u (x,t)\mathrm{d}x\mathrm{d}t,}
$$
and $p_u$ is the unique weak solution of the protection problem~\eqref{modele} for the control $u\in\mathcal{U}_{ad}$.

\medskip

\paragraph{Notations} We denote by $\LL^2(\Omega), \HH^1(\Omega), \HH^1 _0(\Omega)$, the usual Lebesgue and Sobolev spaces endowed
with their standard norms, by $\mathrm{B}_n(0, r)$ the Euclidean open ball of~$\R^n$ centered at $0$ with radius~$r > 0$, and by~$\chi_{\mathrm{C}}$ the characteristic function of $\mathrm{C}\subset\R^n$ defined by~$\chi_{\mathrm{C}} (x)=1$ if~$x\in \mathrm{C}$, and $0$ if~$x\notin \mathrm{C}$. At some points of this paper, we will use the notation~$\R^n=\R^{n_1}\times\R^{n_2}$, if $n\geq2$, where~$n_1\in\N^*$,~$n_2\in\N^*$, such that~$n_1+n_2=n$. In that case $x=(x_1,...,x_{n_1})$ will refer to the variables in~$\R^{n_1}$ and~$y=(y_1,...,y_{n_2})$ to the variables in~$\R^{n_2}$. Finally, we denote $\mathcal{V}_{ad}\subset\mathcal{U}_{ad}$ the set of controls that saturate the integral constraint, i.e.,
\begin{equation}\label{setconstr}
    \mathcal{V}_{ad}:=\left\{ u\in\LL^{\infty}(\Omega) \mid 0\leq u \leq 1 \text{ \textit{a.e.} in } \Omega, \text{ and } \int_\Omega u (x)\,\mathrm{d}x = L \right\},
\end{equation}
and we call bang-bang control an element $u\in\mathcal{V}_{ad}$ such that $u\in\{0,1\}$ \textit{a.e.} in $\Omega$ and $\int_{\Omega}u(x)\,\mathrm{d}x = L$.

\medskip

\paragraph{Main results}

Our main results are summarized in the following theorem.

\begin{myTheorem}\label{maintheo}
     There exists a unique solution to the optimal control problem~\eqref{optimalcontrol} and it is a bang-bang control. Moreover:
    \begin{enumerate}[label={\rm (\roman*)}]
    \item If $\Omega:=\mathrm{B}_n(0,r)$, where $r>0$. Then $u^*:=\chi_{\mathrm{B}_n\left(0,(\frac{L}{\omega_n})^{\frac{1}{n}}\right)}$ is the unique solution of the optimal control problem~\eqref{optimalcontrol}, where $\omega_n:=|\mathrm{B}_n(0,1)|$; \label{firstitem}
    \item If $n\geq 2$, and~$\Omega=\mathrm{B}_{n_1}(0,r_1)\times\Omega_{n_2}\subset\R^{n_1} \times\R^{n_2}$, where $r_1>0$, and $\Omega_{n_2}$ is a nonempty bounded connected open subset of $\R^{n_2}$, then there exists a subset $\mathrm{E}\subset\Omega$ satisfying $|\mathrm{E}|=L$, which is symmetric with respect to the hyperplane~$x_i=0$ and that is convex in the $x_i$-direction, for all $i\in\{1,...,n_1\}$, such that~$u^*=\chi_\mathrm{E}$ is the unique solution of the optimal control problem~\eqref{optimalcontrol};\label{seconditem}
    \item If $n\geq2$, and  $\Omega=\mathrm{B}_{n_1}(0,r_1)\times \mathrm{B}_{n_2}(0,r_2)\subset\R^{n_1} \times\R^{n_2}$, where $(r_1,r_2)\in\R^+ _*\times\R^+ _*$, then there exists a subset $\mathrm{E}\subset\Omega$ satisfying $|\mathrm{E}|=L$, which is symmetric with respect to the hyperplane $x_i=0$ and with respect to the hyperplane $y_j=0$, and that is convex in the $x_i$-direction, convex in the $y_j$-direction, for all~$(i,j)\in\{1,...,n_1\}\times\{1,...,n_2\}$, and star-shaped at $0\in\mathrm{E}$, such that~$u^*=\chi_\mathrm{E}$ is the unique solution of the optimal control problem~\eqref{optimalcontrol}.\label{thirditem}
\end{enumerate}
\end{myTheorem}

\medskip

\paragraph{Symmetrization methods}
In this paper, symmetrization methods are used in order to determine the shape of the optimal intervention for particular crop field geometry. They consist in transforming one set into another with symmetry properties. Two methods are mainly used: the \textit{Schwarz symmetrization} and the \textit{Steiner symmetrization}. In particular, they are used in the literature to compare solutions of partial differential equations, establishing relations between the norm of the original solutions and that of their symmetrized counterparts. The Schwarz symmetrization (see, e.g.,~\cite{Burchard,Henrot,KAW,KESAVAN}) consists in rearranging the level sets of a function
in balls of the same measure in order to create a new function that is symmetric and radially non-increasing. In particular, this method is used to prove the well-known Faber-Krahn inequality (see, e.g,~\cite[Chapter 3, Section 3.2]{Henrot}). Comparison results using Schwarz symmetrization started with the pioneering work~\cite{TAL} where the symmetrized problem is given by the Dirichlet Laplacian operator. This result is extended in~\cite{ALVIN,BANDLE2,Chiti,PLL,TROMBV} and~\cite[Section 9]{Trombetti}) for second-order elliptic operators including lower-order terms. We also refer to~\cite{ALVNIT,AMATO,LANG} for Neumann and Robin boundary conditions and to~\cite{ALVIN, BANDLE,MOSS, VAZ} for parabolic equations. The Steiner symmetrization (see, e.g.,~\cite{Baernstein,Henrot,KAW}) is a partial symmetrization which consists in symmetrizing a set with respect to a hyperplane. In particular, when applied to a level set of a function leads to a new function that is symmetric with respect to a hyperplane. In~\cite{LIONS2} a comparison result is proved for Dirichlet boundary conditions and second-order elliptic operators without lower-order terms. Generalizations can be found, for instance, in~\cite{CHIA2,CHIA}, where the authors deal with the first and zeroth-order term. We also refer to~\cite{FERMER, FENG} for Neumann boundary conditions and to~\cite{DIAGOM} for parabolic equations. We conclude this paragraph by mentioning the works~\cite{BURT,CUCCU,BEH,FERO,VOAS} where the authors studied some optimization problems using symmetrization methods.

\medskip

\paragraph{Summary of the methodology} Our methodology is based on the theory of optimal control and on comparison results for elliptic equations via symmetrization methods.
Precisely, we prove that if there exists a solution to the optimal control problem~\eqref{optimalcontrol}, then it must saturates the integral constraint and, on the other hand, we prove that the cost functional~$\mathcal{J}$ is stricly convex by computing its twice Fréchet differential. Thus, the supremum of $\mathcal{J}$ on~$\mathcal{U}_{ad}$ coincides with the supremum of~$\mathcal{J}$ on the bang-bang controls set. Then, using a suitable bijection, we prove the existence and the uniqueness of a solution among bang-bang controls. For particular domains, we are able to prove some results about the shape of the optimal intervention using symmetrization methods.

\medskip

\paragraph{Paper structure} The paper is organized as follows. In Section~\ref{preliminaries}, we prove all preliminary results required to takle Theorem~\ref{maintheo}. In  Section~\ref{mainproof}, we focus on the proof of Theorem~\ref{maintheo} and, in Section~\ref{addrem}, additional results and remarks are presented. In Section~\ref{num}, numerical simulations are performed in order to illustrate our results. Finally, in Appendix~\ref{remind} some notions and results on Schwarz and Steiner symmetrizations are recalled, and the relation between extreme points and bang-bang controls is reminded in Appendix~\ref{appendixB}.

\section{Preliminary results}\label{preliminaries}

This section is devoted to the statement of some preliminary results needed to prove the main Theorem~\ref{maintheo}. In a nutshell, we first prove here that if solutions to the optimal control problem~\eqref{optimalcontrol} exist, then they must saturate the integral constraint. Thus, our analysis is reduced to the set $\mathcal{V}_{ad}\subset \mathcal{U}_{ad}$ which is the set of controls that saturate the integral constraint (see~\eqref{setconstr}). Then, we prove that the cost functional~$\mathcal{J}$ is strictly convex on $\mathcal{V}_{ad}$ and thus its supremum on $\mathcal{V}_{ad}$ coincides with its supremum on the set of extreme points of $\mathcal{V}_{ad}$, i.e. on bang-bang controls set.

Let us recall that, for all $u\in\mathcal{U}_{ad}$, the existence and uniqueness of the weak solution~$p_u\in\LL^2(]0,T[,\HH^1  (\Omega))$ of the protection problem~\eqref{modele} are guaranteed by the Riesz representation theorem.

\begin{myProp}\label{sature}
Let us assume that there exists $u^*\in\mathcal{U}_{ad}$ such that $\mathcal{J}(u^*)\geq \mathcal{J}(u)$ for all~$u\in\mathcal{U}_{ad}$. Then,
$$
\int_{\Omega} u^* (x)\,\mathrm{d}x=L.
$$
\end{myProp}

\begin{proof}
Assume that $\int_{\Omega} u^* (x)\,\mathrm{d}x<L$. Then there exists $\lambda>0$ such that $\int_{\Omega} (u^*+\lambda)=L$. Let us denote~$u_{\lambda} := u^*+\lambda \chi_{\{x\in\Omega | u^*(x) +\lambda \leq 1\}} \in\mathcal{U}_{ad}$, then $u_\lambda\geq u^*$ \text{a.e.} in $\Omega$ and $u_\lambda>u^*$ on a set of positive measure. 
Thus, since $\phi(u^*)\leq\phi(u_\lambda)$, one deduces from the maximum principle that~$p_{u_\lambda} \geq p_{u^*}$ \textit{a.e.} in $\Omega\times]0,T[$. 
Hence $\mathcal{J}(u_{\lambda})\geq\mathcal{J}(u^*)$. It follows that $\mathcal{J}(u_{\lambda})=\mathcal{J}(u^*)$, therefore $p_{u_\lambda}=p_{u^*}$ \textit{a.e.} in~$\Omega\times]0,T[$ and one deduces that $u_\lambda = u^*$ \textit{a.e.} in~$\Omega\times]0,T[$ which is a contradiction.
\end{proof}

From Proposition~\ref{sature}, it follows that 
\begin{equation*}
            \sup\limits_{ \substack{ u\in \mathcal{U}_{ad} \\  } } \; \mathcal{J}(u)=\sup\limits_{ \substack{ u\in \mathcal{V}_{ad} \\  } } \; \mathcal{J}(u).
\end{equation*}

Now let us prove that $\mathcal{J}$ is Fréchet differentiable and twice Fréchet differentiable in~$\LL^{\infty}(\Omega)$.

\begin{myProp}\label{differentielle}
For all $u\in\LL^{\infty}(\Omega)$, the cost functional $\mathcal{J}$ is Fréchet differentiable at~$u$ and its differential~$ \mathrm{d}\mathcal{J}(u)$ is given by
\begin{equation*}
 \mathrm{d}\mathcal{J}(u)\left(h\right)=\int_{\Omega}h(x) q(x) w\circ u (x) \,\mathrm{d}x , \qquad\forall h\in\LL^\infty(\Omega),
\end{equation*}
where $q\in\HH^1 _0(\Omega)$ denotes the adjoint state being the unique weak solution to the Dirichlet problem
\begin{equation}\label{adjoint}
\arraycolsep=2pt
\left\{
\begin{array}{rcll}
  q - D\Delta q & = & 1 & \text{ in } \Omega , \\
 q & = & 0  & \text{ on } \partial{\Omega},\\
\end{array}
\right.
\end{equation}
and $w\in\LL^{\infty}(\R)$ is the switching function given by 
\begin{equation*}
\displaystyle\fonction{w}{\R}{\R^+}{s}{\displaystyle \left\{
\begin{array}{ccll}
\alpha \frac{T^2}{2} 	&   & \text{ if } s=1, \\ [5pt]
\frac{\left(1+\mathrm{e}^{-\alpha\left(1-s\right)T}\left(-\alpha(1-s)T-1\right)\right)}{{\alpha\left(1-s\right)^2}}	&   & \text{ if } s\neq1.
\end{array}
\right.}
\end{equation*}
\end{myProp}

\begin{proof}
By standard computations, the  map $u\in\LL^{\infty}(\Omega)\mapsto p_{u}\in\LL^2(]0,T[,\HH^1 (\Omega))$ is Fréchet differentiable at every~$u\in\LL^{\infty}(\Omega)$ with its differential at $u$ given, for all $h\in\LL^{\infty}(\Omega)$, by $\dot{p}_u (h)\in \LL^2(]0,T[,\HH^1(\Omega))$ solution of
\begin{equation*}
\arraycolsep=2pt
\left\{
\begin{array}{rcll}
  \dot{p}_u (h)- D\Delta \dot{p}_u (h) & = &  \alpha h t \mathrm{e}^{-\alpha \left(1-u\right)t} & \text{ in } \Omega\times]0,T[ , \\
\dot{p}_u (h) & = & 0  & \text{ on } \partial{\Omega}\times]0,T[.
\end{array}
\right.
\end{equation*}
Then, it follows that $\mathcal{J}$ is Fréchet differentiable at any $u\in\LL^{\infty}(\Omega)$ with its differential at $u$ given by
$$
\mathrm{d}\mathcal{J}(u)(h)=\int_0^T\int_{\Omega} \dot{p}_u (h)(x,t)\,\mathrm{d}x\, \mathrm{d}t, \qquad \forall h\in\LL^{\infty}(\Omega).
$$
For all $(u,h)\in\LL^{\infty}(\Omega)\times\LL^{\infty}(\Omega)$, let us take the adjoint state $q\in\HH^1 _0(\Omega)$, solution to Problem~\eqref{adjoint}, as test function in the weak variational formulation of $\dot{p}_u(h)$ and, for almost all $t\in]0,T[$,~$\dot{p}_u (h)(t,\cdot)\in\HH^1 _0(\Omega)$ as test function in the weak variational formulation of~$q$. Then, one can deduce that
$$
\int_0^T\int_{\Omega} \dot{p}_u(h)(x,t)\mathrm{d}x \mathrm{d}t=\int_0 ^T \int_{\Omega} \alpha h(x) t \mathrm{e}^{-\alpha \left(1-u(x)\right)t}  q(x) \mathrm{d}x \mathrm{d}t.
$$
We define the switching function by $w(s):=\alpha\int_0 ^T t\mathrm{e}^{-\alpha \left(1-s\right)t}\mathrm{d}t\in\R^+$, for almost all $s\in\R$, thus
$$
\mathrm{d}\mathcal{J}(u)(h)=\int_{\Omega}h(x) q(x) w\circ u (x) \mathrm{d}x,
$$
which concludes the proof.
\end{proof}

\begin{myProp}
    For all $u\in\LL^{\infty}(\Omega)$, the cost functional $\mathcal{J}$ is twice Fréchet differentiable at~$u$ and its second differential $\mathrm{d}^2 \mathcal{J}(u)$ is given by 
$$
\mathrm{d}^2 \mathcal{J}(u)\left(h_1,h_2\right)=\int_{\Omega}\alpha^2 q(x) h_1 (x) h_2 (x)  \left(\int_{0}^{T}t^2\mathrm{e}^{-\alpha(1-u(x))t} \mathrm{d}t\right)\mathrm{d}x, \quad\forall \left(h_1,h_2\right)\in\LL^\infty(\Omega)\times\LL^\infty(\Omega).
$$
Moreover, for all $u\in\LL^{\infty}(\Omega)$, $\mathrm{d}^2 \mathcal{J}(u)$ is positive definite thus $\mathcal{J}$ is a strictly convex function on~$\LL^{\infty}(\Omega)$.
\end{myProp}

\begin{proof}
One can easily prove that the map $u\in\LL^{\infty}(\Omega)\mapsto p_{u}\in\LL^2(]0,T[,\HH^1(\Omega))$  is twice Fréchet differentiable at every~$u\in\LL^{\infty}(\Omega)$ with its second differential at $u$ given, for all $(h_1,h_2)\in\LL^{\infty}(\Omega)\times \LL^{\infty}(\Omega)$, by $\ddot{p}_u (h_1,h_2)\in \LL^2(]0,T[,\HH^1(\Omega))$ being the unique weak solution of
\begin{equation*}
\arraycolsep=2pt
\left\{
\begin{array}{rcll}
  \ddot{p}_u (h_1,h_2)- D\Delta \ddot{p}_u (h_1,h_2) & = &  \alpha^2  h_1 h_2 t^2 \mathrm{e}^{-\alpha \left(1-u\right)t} & \text{ in } \Omega\times]0,T[ , \\
\ddot{p}_u (h_1,h_2) & = & 0  & \text{ on } \partial{\Omega}\times]0,T[.
\end{array}
\right.
\end{equation*}
Thus, it can be deduced that $\mathcal{J}$ is twice Fréchet differentiable at any $u\in\LL^{\infty}(\Omega)$ with its second differential at $u$ given by
$$
\mathrm{d}^2\mathcal{J}(u)(h_1,h_2)=\int_0^T\int_{\Omega} \ddot{p}_u (h_1,h_2)(x,t)\mathrm{d}x \mathrm{d}t, \qquad \forall \left(h_1,h_2\right)\in\LL^{\infty}(\Omega)\times \LL^{\infty}(\Omega).
$$
Then, using the adjoint state $q\in\HH^1 _0(\Omega)$ one can prove in the same way as in Proposition~\ref{differentielle}, that 
$$
\int_0^T\int_{\Omega} \ddot{p}_u (h_1,h_2)(x,t)\mathrm{d}x \mathrm{d}t=\int_{\Omega}\alpha^2 q(x) h_1 (x) h_2 (x) \left(  \int_{0}^{T}t^2\mathrm{e}^{-\alpha(1-u(x))t}\mathrm{d}t\right)\mathrm{d}x,
$$
for all $u,h_1,h_2\in\LL^\infty(\Omega)$. Finally, by considering $h_2=h_1$, one gets 
$$
\mathrm{d}^2\mathcal{J}(u)\left(h_1,h_1\right)=\int_{\Omega}\alpha^2 q(x) h_1 (x) ^2 \left( \int_{0}^{T}t^2\mathrm{e}^{-\alpha(1-u(x))t}\mathrm{d}t\right) \mathrm{d}x.
$$
From the strong maximum principle one has $q>0$ \textit{a.e.} in $\Omega$, thus $\mathrm{d}^2\mathcal{J}(u)\left(h_1,h_1\right)\geq0$ with equality if and only if $h_1=0$ \textit{a.e.} in $\Omega$. Thus $\mathrm{d}^2\mathcal{J}(u)$ is positive definite and, since $u\in\LL^{\infty}(\Omega)$ is arbitrary, one concludes that $\mathcal{J}$ is a stricly convex function.
\end{proof}

Since $\mathcal{J}$ is stricly convex on $\mathcal{V}_{ad}$, it is well known that
\begin{equation*}
            \sup\limits_{ \substack{ u\in \mathcal{V}_{ad} \\  } } \; \mathcal{J}(u)=\sup\limits_{ \substack{ u\in \mathrm{Extr}(\mathcal{V}_{ad}) \\  } } \; \mathcal{J}(u), 
\end{equation*}
where $\mathrm{Extr}(\mathcal{V}_{ad})$ is the set of extreme points of $\mathcal{V}_{ad}$ defined by
$$
\mathrm{Extr}(\mathcal{V}_{ad}):=\left\{ u\in\mathcal{V}_{ad} \mid  v_1,v_2\in\mathcal{V}_{ad}, u=\frac{v_1+v_2}{2} \implies u=v_1=v_2 \right\},
$$
which corresponds to the set of bang-bang controls (see Proposition~\ref{ProofExtrBang}),
$$
\mathrm{Extr}(\mathcal{V}_{ad})=\left\{ u\in\LL^{\infty}(\Omega) \mid u\in\left\{0,1\right\} \text{ \textit{a.e.} in } \Omega, \text{ and } \int_{\Omega}u(x)\mathrm{d}x= L\right\}.
$$

\section{Proof of Theorem~\ref{maintheo}}\label{mainproof}
In this section, we prove Theorem~\ref{maintheo}. From the previous section, we know that if there exists a control which maximizes $\mathcal{J}$ on the bang-bang controls set, then it is also solution to the optimal control problem~\eqref{optimalcontrol}. Thus, inspired by~\cite{MAZ}, let us prove that there exists a unique solution among bang-bang controls.

\begin{proof}[Proof of the existence and uniqueness of a bang-bang solution]
For all $u\in\mathrm{Extr}(\mathcal{V}_{ad})$, let us consider $P_u:=\int_{0}^{T} p_u (\cdot,t)\mathrm{d}t /T  \in \HH^1(\Omega)$. Then $P_u$ is the unique weak solution of the Dirichlet problem
\begin{equation*}
\arraycolsep=2pt
\left\{
\begin{array}{rcll}
  P_u - D\Delta P_u & = &  \displaystyle \frac{1}{T}\int_{0}^T \mathrm{e}^{-\alpha(1-u)t}\mathrm{d}t & \text{ in } \Omega , \\
P_u & = & 1  & \text{ on } \partial{\Omega}.
\end{array}
\right.
\end{equation*}
Denote $\varphi(u):=  \int_{0}^T \mathrm{e}^{-\alpha(1-u)t}\mathrm{d}t /T \in\LL^{\infty}(\Omega)$, for all $u\in\mathrm{Extr}(\mathcal{V}_{ad})$. Therefore,
\begin{equation*}
\arraycolsep=2pt
\varphi (u):=\left\{
\begin{array}{rl}
  1 &    \text{ if } u=1 , \\
 \frac{1-\mathrm{e}^{-\alpha T}}{\alpha T} &    \text{ if } u=0,
\end{array}
\right.
\end{equation*}
and
$$
\int_{\Omega}\varphi(u)(x)\mathrm{d}x=\int_{\{u=1\}}\mathrm{d}x+\int_{\{u=0\}}\frac{1-\mathrm{e}^{-\alpha T}}{\alpha T}\mathrm{d}x=L+\frac{1-\mathrm{e}^{-\alpha T}}{\alpha T}\left(|\Omega|-L\right)=:M.
$$
Now consider
\begin{equation*}
    \mathcal{F} :=\left\{ f\in\LL^{\infty}(\Omega) \mid f\in\left\{\frac{1-\mathrm{e}^{-\alpha T}}{\alpha T}, 1\right\} \text{ \textit{a.e.} in } \Omega, \text{ and } \int_{\Omega}f(x)\mathrm{d}x= M \right\},
\end{equation*}

Then $\varphi$ is a bijective function from $\mathrm{Extr}(\mathcal{V}_{ad})$ to $\mathcal{F}$ with its inverse given by
\begin{equation*}
\arraycolsep=2pt
\varphi^{-1} (f):=\left\{
\begin{array}{rl}
  1 &    \text{ if } f=1 , \\
 0 &    \text{ if } f=\frac{1-\mathrm{e}^{-\alpha T}}{\alpha T},
\end{array}
\right.
\end{equation*}
for all $f\in\mathcal{F}$. 

Let us introduce the control problem given by
\begin{equation}\label{relaxpb}
            \max\limits_{ \substack{ f\in \mathcal{F} \\  } } \; J(f),
\end{equation}
where $J : f\in\mathcal{F} \mapsto \int_{\Omega} Z_f (x)\mathrm{d}x \in\R$, and where $Z_f\in\HH^1 _0(\Omega)$ is the unique weak solution of
\begin{equation}\label{relaxZ}
\arraycolsep=2pt
\left\{
\begin{array}{rcll}
  Z_f - D\Delta Z_f & = &  f & \text{ in } \Omega , \\
Z_f & = & 0  & \text{ on } \partial{\Omega}.
\end{array}
\right.
\end{equation}

Consider the adjoint state $q\in\HH^1 _0(\Omega)$, solution to Problem~\eqref{adjoint}, as test function in the weak variational formulation of $Z_f\in\HH^1 _0(\Omega)$, and $Z_f\in\HH^1 _0(\Omega)$ as test function in the weak variational formulation of $q\in\HH^1 _0(\Omega)$,  for all $f\in\mathcal{F}$. Then
one deduces that
$$
\int_\Omega Z_f (x)\mathrm{d}x = \int_\Omega q(x)f(x)\mathrm{d}x , \qquad \forall f\in\mathcal{F}.
$$
Now, let $\mathrm{E}:=\{x\in\Omega \mid q(x) \geq \xi \}$ where $\xi\geq0$ is chosen such that $|\mathrm{E}|=L$, and denote~$f^*:=\chi_\mathrm{E} +((1-\mathrm{e}^{-\alpha T})/\alpha T) \chi_{\Omega\backslash \mathrm{E}}$. Consider also $f\in\mathcal{F}$, then there exists $\mathrm{C}\subset\Omega$, $|\mathrm{C}|=L$ such that~$f=\chi_\mathrm{C} +((1-\mathrm{e}^{-\alpha T})/\alpha T) \chi_{\Omega\backslash \mathrm{C}}$. It follows that $\mathrm{C}=(\mathrm{C}\cap \mathrm{E}) \cup (\mathrm{C}\cap \Omega\backslash \mathrm{E})$ and $\mathrm{E}=(\mathrm{E}\cap \mathrm{C}) \cup (\mathrm{E}\cap \Omega\backslash \mathrm{C})$ where
$\mathrm{C}\cap \Omega\backslash \mathrm{E}=\{x\in\mathrm{C} \mid q(x) < \xi \}$ and $\mathrm{E}\cap \Omega\backslash \mathrm{C}=\{x\in\Omega\backslash\mathrm{C} \mid q(x) \geq \xi \}$. Note that, since~$|\mathrm{E}|=|\mathrm{C}|$, then $|\mathrm{E}\cap \mathrm{C}|+|\mathrm{E}\cap \Omega\backslash \mathrm{C}|=|\mathrm{C}\cap \mathrm{E}| + |\mathrm{C}\cap \Omega\backslash \mathrm{E}|$, thus $|\mathrm{E}\cap \Omega\backslash \mathrm{C}|=|\mathrm{C}\cap \Omega\backslash \mathrm{E}|$. One has
\begin{multline*}
    \int_\Omega q(x)f(x)\mathrm{d}x=\int_\Omega q(x)\left(f(x)-\frac{1-\mathrm{e}^{-\alpha T}}{\alpha T}\right)\mathrm{d}x+\left(\frac{1-\mathrm{e}^{-\alpha T}}{\alpha T}\right)\int_\Omega q(x)\mathrm{d}x \\=\left(1-\frac{1-\mathrm{e}^{-\alpha T}}{\alpha T}\right)\int_\Omega q(x)\chi_{\mathrm{C}}(x)\mathrm{d}x+\left(\frac{1-\mathrm{e}^{-\alpha T}}{\alpha T}\right)\int_\Omega q(x)\mathrm{d}x.
\end{multline*}
Moreover, since $q<\xi$ on $\mathrm{C}\cap \Omega\backslash \mathrm{E}$ and $|\mathrm{E}\cap \Omega\backslash \mathrm{C}|=|\mathrm{C}\cap \Omega\backslash \mathrm{E}|$, we have
\begin{multline*}
    \int_{\mathrm{C}} q(x)\mathrm{d}x  \leq \int_{\mathrm{C}\cap\mathrm{E}} q(x)\mathrm{d}x  + \xi \left|\mathrm{C}\cap \Omega\backslash \mathrm{E}\right|=\int_{\mathrm{C}\cap\mathrm{E}} q(x)\mathrm{d}x  + \int_{\mathrm{E}\cap \Omega\backslash \mathrm{C}}\xi\mathrm{d}x\\ \leq \int_{\mathrm{C}\cap\mathrm{E}} q(x)\mathrm{d}x  + \int_{\mathrm{E}\cap \Omega\backslash \mathrm{C}}q(x)\mathrm{d}x = \int_{\mathrm{E}}q(x)\mathrm{d}x  , 
\end{multline*}
with equality if and only if $|\mathrm{C}\cap \Omega\backslash \mathrm{E}|=0$, i.e., if and only if $\chi_{\mathrm{C}}=\chi_{\mathrm{E}}$ \textit{a.e.} on $\Omega$. Thus, one deduces
\begin{multline*}
    \int_\Omega Z_f (x)\mathrm{d}x=\int_\Omega q(x)f(x)\mathrm{d}x\leq\left(1-\frac{1-\mathrm{e}^{-\alpha T}}{\alpha T}\right)\int_\Omega q(x)\chi_{\mathrm{E}}(x)\mathrm{d}x+\left(\frac{1-\mathrm{e}^{-\alpha T}}{\alpha T}\right)\int_\Omega q(x)\mathrm{d}x\\=\int_{\Omega}q(x)f^*(x)\mathrm{d}x=\int_\Omega Z_f^* (x)\mathrm{d}x,
\end{multline*}
with equality if and only if $\chi_{\mathrm{C}}=\chi_{\mathrm{E}}$ \textit{a.e.} on $\Omega$, i.e., if and only $f=f^*$ \textit{a.e.} on $\Omega$. Hence, $f^*$ is the unique solution of the control problem~\eqref{relaxpb}.

Moreover, one has $P_u=Z_{\varphi(u)}+W$, for all $u\in\mathrm{Extr}(\mathcal{V}_{ad})$, where~$W\in\HH^1  (\Omega)$ is the unique weak solution of
\begin{equation}\label{annexeprob2}
\arraycolsep=2pt
\left\{
\begin{array}{rcll}
  W - D\Delta W & = &  0 & \text{ in } \Omega , \\
W & = & 1  & \text{ on } \partial{\Omega}.
\end{array}
\right.
\end{equation}
Therefore, by considering $u^*:=\varphi^{-1}({f^*})=\chi_{\mathrm{E}}\in\mathrm{Extr}(\mathcal{V}_{ad})$, it follows that
\begin{multline*}
    \int_{\Omega} P_{u^*}(x)\mathrm{d}x=\int_{\Omega} Z_{\varphi(u^*)} (x)\mathrm{d}x+\int_{\Omega}W(x)\mathrm{d}x\\\geq\int_{\Omega} Z_{\varphi(u)} (x)\mathrm{d}x+\int_{\Omega}W(x)\mathrm{d}x=\int_{\Omega} P_{u}(x)\mathrm{d}x, \qquad \forall u\in\mathrm{Extr}(\mathcal{V}_{ad}).
\end{multline*}
Thus, by definition of $P_u$, one gets
$$
\int_0 ^T \int_{\Omega} p_{u^*} (x,t)\mathrm{d}x \mathrm{d}t\geq \int_0 ^T \int_{\Omega} p_{u} (x,t)\mathrm{d}x \mathrm{d}t, \qquad \forall u\in\mathrm{Extr}(\mathcal{V}_{ad}),
$$
i.e., 
$$
\mathcal{J}(u^*)=\max\limits_{ \substack{ u\in \mathrm{Extr}(\mathcal{V}_{ad}) \\  } } \; \mathcal{J}(u), 
$$
and since $f^*$ is unique, it follows that $u^*$ is unique, which concludes the proof.
\end{proof}

\begin{myRem}\normalfont
Note that a nonhomogeneous Dirichlet condition in problem~\eqref{relaxZ} could be considered in order to avoid introducing $W\in\HH^1(\Omega)$ solution to problem~\eqref{annexeprob2}. Nevertheless, the homogeneous Dirichlet condition will be required to prove items~\ref{firstitem},\ref{seconditem} and \ref{thirditem} of Theorem~\ref{maintheo}.
\end{myRem}

Now let us prove item~\ref{firstitem} using the Schwarz symmetrization which is recalled in Appendix~\ref{Schwarz}.

\begin{proof}[Proof of item~\ref{firstitem}]
In the proof of the existence/uniqueness of a bang-bang solution, we deduced that there exists~$\mathrm{E}\in\Omega$, $|\mathrm{E}|=L$, such that~$f^*=\chi_\mathrm{E} +\frac{(1-\mathrm{e}^{-\alpha T})}{\alpha T} \chi_{\Omega\backslash \mathrm{E}}$ is the unique solution satisfying
$$
\int_{\Omega} Z_{f} (x)\mathrm{d}x\leq\int_{\Omega} Z_{f^*} (x)\mathrm{d}x, \qquad\forall f\in\mathcal{F},
$$
where $Z_{f}\in\HH^1 _0 (\Omega)$ is the unique weak solution of problem~\eqref{relaxZ} for $f\in\mathcal{F}$. Therefore, using the Schwarz symmetrization comparison theorem (see Theorem~\ref{theoAlv}), one deduces that,
\begin{equation}\label{inegali}
    \int_{\Omega} Z_{f^*}(x) \mathrm{d}x \leq \int_{\Omega^{\#}} v(x) \mathrm{d}x,
\end{equation}
where~$v\in\HH^1 _0(\Omega^{\#})$ is the unique weak solution of
\begin{equation*}
\arraycolsep=2pt
\left\{
\begin{array}{rcll}
  v - D\Delta v & = &  {f^*}^{\#} & \text{ in } \Omega^{\#} , \\
v & = & 0  & \text{ on } \partial{\Omega}^{\#},
\end{array}
\right.
\end{equation*}
and where $\Omega^\#$ is the symmetric rearrangement of $\Omega$ and ${f^*}^{\#}=\chi_{\mathrm{E}^{\#}} +((1-\mathrm{e}^{-\alpha T})/\alpha T)\chi_{\Omega^{\#}\backslash \mathrm{E}^{\#}}$ is the Schwarz symmetrization of $f^*$ (see~Example~\ref{indicSch}). Moreover, since $\Omega=\mathrm{B}_n (0,r)$, then $\Omega^{\#}=\Omega$, ${f^*}^{\#}\in\mathcal{F}$ and $v=Z_{{f^*}^{\#}}$. Thus, by Inequality~\eqref{inegali}, one has,
\begin{equation*}
 \int_{\Omega} Z_{f^*}(x)\mathrm{d}x=\int_{\Omega} Z_{{f^*}^{\#}}(x)\mathrm{d}x\geq\int_{\Omega} Z_{f} (x)\mathrm{d}x, \qquad\forall f\in\mathcal{F}.
\end{equation*}
Since the solution of the control problem~\eqref{relaxpb} is unique, it follows that $f^*={f^*}^{\#}$, and that $u^*:=\varphi^{-1}({{f^*}^{\#}})=\chi_{\mathrm{E}^{\#}}=\chi_{\mathrm{B}_n\left(0,(\frac{L}{\omega_n})^{\frac{1}{n}}\right)}$ is the unique solution of the optimal control problem~\eqref{optimalcontrol}.
\end{proof}

\begin{myRem}\normalfont
Note that, in the proof of the existence/uniqueness of a bang-bang solution, we have seen that the unique solution of the optimal control problem~\eqref{optimalcontrol} is given by $u^*=\chi_{\mathrm{E}}$ where~$\mathrm{E}:=\{ x\in\Omega \mid q(x)\geq\xi \}$, for $\xi\geq0$ such that $|\mathrm{E}|=L$. Therefore, since $q\in\HH^1_0 (\Omega)$ has the same symmetries as $\Omega$, then so does the set $\mathrm{E}$. In particular, if~$\Omega=\mathrm{B}_n (0,r)$, then~$q$ is symmetric with respect to any hyperplane and, one can prove in the same way as~\cite[Part I Chapter 2 section 2.4]{TRU}, that $q\in\HH^1_0(\Omega)$ is radially non-increasing. Thus $\mathrm{E}$ is a ball centered at $0$ with measure equal to~$L$.
\end{myRem}

To prove items~\ref{seconditem} and~\ref{thirditem} some properties and results on the Steiner symmetrization are used, and we refer to Appendix~\ref{Steiner}, for more details.

\begin{proof}[Proof of item~\ref{seconditem}]
The proof is close to the one of item~\ref{firstitem}. One knows that there exists~$\mathrm{E}\subset \Omega$, $|\mathrm{E}|=L$ such that
$
f^*=\chi_\mathrm{E} +\frac{(1-\mathrm{e}^{-\alpha T})}{\alpha T} \chi_{\Omega\backslash \mathrm{E}},
$
and $f^*$ is the unique solution of the control problem~\eqref{relaxpb}. Since $f^*\in\LL^{\infty}(\Omega)$ and $\Omega$ is bounded, it follows from Theorem~\ref{theoSt}, that,
\begin{equation*}
    \int_{\mathrm{B}_{n_1}(0,r_1)} Z_{f^*} (x,y) \mathrm{d}x \leq \int_{\mathrm{B}_{n_1}(0,r_1)} v(x,y)\mathrm{d}x,
\end{equation*}
for almost all $y\in\Omega_{n_2}$, where $Z_{f^*}\in\HH^1 _0 (\Omega)$ is the unique weak solution of Problem~\eqref{relaxZ} for the control~$f^*$, and~$v\in\HH^1 _0(S_{(n_1,\cdot)}(\Omega))$ is the unique weak solution of
\begin{equation*}
\arraycolsep=2pt
\left\{
\begin{array}{rcll}
  v -D\Delta v & = & S_{(n_1,\cdot)}(f^*)  & \text{ in } S_{(n_1,\cdot)}(\Omega) , \\
v & = & 0  & \text{ on } \partial{S_{(n_1,\cdot)}(\Omega)},
\end{array}
\right.
\end{equation*}
where $S_{(n_1,\cdot)}(\Omega)=\mathrm{B}_{n_1}(0,r_1)\times\Omega_{n_2}$ is the $(n_1,\cdot)$-Steiner symmetrization of $\Omega$ (see Example~\ref{steinercart}) and where 
$$
S_{(n_1,\cdot)}(f^*)=\chi_{S_{(n_1,\cdot)}(\mathrm{E})} + ((1-\mathrm{e}^{-\alpha T})/\alpha T)\chi_{S_{(n_1,\cdot)}(\Omega)\backslash S_{(n_1,\cdot)}(\mathrm{E})},
$$
is the $(n_1,\cdot)$-Steiner symmetrization of $f^*$ (see Example~\ref{indicSteiner}). Furthermore, $| S_{(n_1,\cdot)}(\mathrm{E})|=|\mathrm{E}|$ and~$S_{(n_1,\cdot)}(\Omega)=\Omega$, thus~$S_{(n_1,\cdot)}(f^*)\in\mathcal{F}$ and $v=Z_{S_{(n_1,\cdot)}(f^*)}$. Hence,
$$
\int_{\Omega_{n_2}}\int_{\mathrm{B}_{n_1}(0,r_1)} Z_{f^*} (x,y) \mathrm{d}x \mathrm{d}y\\= \int_{\Omega_{n_2}}\int_{\mathrm{B}_{n_1}(0,r_1)} Z_{S_{(n_1,\cdot)}(f^*)} (x,y) \mathrm{d}x \mathrm{d}y,
$$
and, since the solution of the control problem~\eqref{relaxpb} is unique, one concludes that $f^*=S_{(n_1,\cdot)}(f^*)$ and that $u^*=\chi_{S_{(n_1,\cdot)}(\mathrm{E})}\in\mathrm{Extr}(\mathcal{V}_{ad})$ is the unique solution of the optimal control problem~\eqref{optimalcontrol}. The symmetry of $S_{(n_1,\cdot)}(\mathrm{E})$ with respect to the hyperplane $x_i=0$ and the convexity in the $x_i$-direction for all $i\in\{1,...,n_1\}$, follow from Proposition~\ref{propstei}.
\end{proof}

\begin{proof}[Proof of item~\ref{thirditem}]
In the proof of item~\ref{seconditem}, one constructs 
$$
S_{(n_1,\cdot)}(f^*)=\chi_{S_{(n_1,\cdot)}(\mathrm{E})} +\frac{1-\mathrm{e}^{-\alpha T}}{\alpha T}\chi_{\Omega\backslash S_{(n_1,\cdot)}(\mathrm{E})},
$$ the unique solution of the control problem~\eqref{relaxpb}. Therefore, by considering $S_{(n_2,\cdot)}(\Phi(\Omega))$ and $Z_{S_{(n_1,\cdot)}(f^*)}\circ \Phi^{-1}$, where $\Phi: (x,y)\in\R^{n_1}\times\R^{n_2}\mapsto(y,x)\in\R^{n_2}\times\R^{n_1}$, it follows from Theorem~\ref{theoSt} that
$$
\int_{\mathrm{B}_{n_2}(0,r_2)} Z_{S_{(n_1,\cdot)}(f^*)}(x,y) \mathrm{d}y\leq \int_{\mathrm{B}_{n_2}(0,r_2)} v(x,y) \mathrm{d} y,
$$
for almost all $x\in\mathrm{B}_{n_1}(0,r_1)$, where $v\in\HH_0^1(S_{(\cdot,n_2)}(\Omega))$ is the unique weak solution of \begin{equation*}
\arraycolsep=2pt
\left\{
\begin{array}{rcll}
  v -D\Delta v & = &  S_{(\cdot,n_2)}\left(S_{(n_1,\cdot)}(f^*)\right)  & \text{ in } S_{(\cdot,n_2)}(\Omega) , \\
v & = & 0  & \text{ on } \partial{S_{(\cdot,n_2)}(\Omega)},
\end{array}
\right.
\end{equation*}
where $S_{(\cdot,n_2)}(\Omega):=\Phi^{-1}\left(S_{(n_2,\cdot)}(\Phi(\Omega))\right)=\mathrm{B}_{n_1}(0,r_1)\times\mathrm{B}_{n_2}(0,r_2)$ is the $(\cdot,n_2)$-Steiner symmetrization of~$\Omega$, and~$S_{(\cdot,n_2)}\left(S_{(n_1,\cdot)}(f^*)\right):=S_{(n_2,\cdot)}\left(S_{(n_1,\cdot)}(f^*)\circ\Phi^{-1}\right)\circ\Phi$ is the $(\cdot,n_2)$-Steiner symmetrization of $S_{(n_1,\cdot)}(f^*)$. One deduces that
$$
S_{(\cdot,n_2)}\left(S_{(n_1,\cdot)}(f^*)\right)=\chi_{S_{(\cdot,n_2)}\left(S_{(n_1,\cdot)}(\mathrm{E})\right)}+\frac{1-\mathrm{e}^{-\alpha T}}{\alpha T}\chi_{S_{(\cdot,n_2)}(\Omega)\backslash S_{(\cdot,n_2)}\left(S_{(n_1,\cdot)}(\mathrm{E})\right)}.
$$
From Proposition~\ref{propstei}, one has $|S_{(\cdot,n_2)}\left(S_{(n_1,\cdot)}(\mathrm{E})\right)|=|S_{(n_1,\cdot)}(\mathrm{E})|=|\mathrm{E}|$ and, since $S_{(\cdot,n_2)}(\Omega)=\Omega$, then~$S_{(\cdot,n_2)}\left(S_{(n_1,\cdot)}(f^*)\right)\in\mathcal{F}$ and~$v=Z_{S_{(\cdot,n_2)}\left(S_{(n_1,\cdot)}(f^*)\right)}$. Thus, one can conclude that $u^*=\chi_{S_{(\cdot,n_2)}\left(S_{(n_1,\cdot)}(\mathrm{E})\right)}\in\mathrm{Extr}(\mathcal{V}_{ad})$ is the unique solution of the optimal control problem~\eqref{optimalcontrol}. From Proposition~\ref{properties} it follows that $S_{(\cdot,n_2)}\left(S_{(n_1,\cdot)}(\mathrm{E})\right)$ is symmetric with respect to the hyperplane $x_i=0$ and with respect to the hyperplane $y_j=0$, and it is a convex set in the $x_i$-direction and in the~$y_j$-direction, for all~$(i,j)\in\{1,...,n_1\}\times\{1,...,n_2\}$, and also star-shaped at $0\in S_{(\cdot,n_2)}\left(S_{(n_1,\cdot)}(\mathrm{E})\right)$.
\end{proof}

\section{Additional results and remarks}\label{addrem}
In this section some additional results and remarks are given. Among domains with the same measure, it is possible to determine on which domain and for which bang-bang control the cost~$\mathcal{J}$ is the lowest.

\begin{myProp}\label{worstfield}
 Let $\mathrm{A}$ be the annulus $\mathrm{A}:=\Omega^{\#}\textbackslash \mathrm{B}_{n}(0,(\frac{|\Omega|-L}{\omega_n})^{\frac{1}{n}})$. Then,
$$
\inf\limits_{ \substack{ u\in \mathrm{Extr}(\mathcal{V}_{ad}) \\  } } \; \mathcal{J}(u) \geq \int_{0}^T \int_{\Omega^ {\#}} p^{\Omega^{\#}}_{\chi_{\mathrm{A}}} (x,t)\mathrm{d}x \mathrm{d}t,
$$
where $p^{\Omega^{\#}}_{\chi_{\mathrm{A}}}\in\LL^2(]0,T[,\HH^1(\Omega^{\#}))$ is the unique weak solution of the protection problem~\eqref{modele} defined on $\Omega^{\#}\times]0,T[$ for the control $u=\chi_{\mathrm{A}}$.
\end{myProp}

\begin{proof}
Let $u\in\mathrm{Extr}(\mathcal{V}_{ad})$ and $W_u:=1-\int_{0}^{T} p_u /T  \in \HH^1(\Omega)$ which is the unique weak solution~of
\begin{equation*}
\arraycolsep=2pt
\left\{
\begin{array}{rcll}
  W_u - D\Delta W_u & = &  \displaystyle 1-\frac{1}{T}\int_{0}^T \mathrm{e}^{-\alpha(1-u)t}\mathrm{d}t & \text{ in } \Omega , \\
W_u & = & 0  & \text{ on } \partial{\Omega}.
\end{array}
\right.
\end{equation*}
Since $u\in\mathrm{Extr}(\mathcal{V}_{ad})$, there exists $\mathrm{C}\subset\Omega$, $|\mathrm{C}|=L$ such that 
$$
1-\frac{1}{T}\int_{0}^T \mathrm{e}^{-\alpha(1-u)t}\mathrm{d}t=\left(1-\frac{1-\mathrm{e}^{-\alpha T}}{\alpha T}\right)\chi_{\Omega\textbackslash \mathrm{C}},
$$
\textit{a.e.} in $\Omega$.
Therefore, by Theorem~\ref{theoAlv}, one deduces that 
$
\int_{\Omega} W_u (x)\mathrm{d}x \leq \int_{\Omega^{\#}} v(x)\mathrm{d}x,
$
where~$v\in\HH^1 _0(\Omega^{\#})$ is the unique weak solution of
\begin{equation*}
\arraycolsep=2pt
\left\{
\begin{array}{rcll}
  v -D\Delta v & = & \left(1-\frac{1-\mathrm{e}^{-\alpha T}}{\alpha T}\right)\chi_{\left(\Omega\textbackslash \mathrm{C}\right)^{\#}}   & \text{ in } \Omega^{\#} , \\
v & = & 0  & \text{ on } \partial{\Omega^{\#}} ,
\end{array}
\right.
\end{equation*}
i.e., 
$$
\int_{\Omega} \mathrm{d}x - \frac{1}{T}\int_{\Omega} \int_0 ^T p_u (x,t)\mathrm{d}x \mathrm{d}t\leq \int_{\Omega^{\#}}  \mathrm{d}x - \int_{\Omega^{\#}} v_1 (x) \mathrm{d}x,
$$
where~$v_1\in\HH^1 (\Omega^{\#})$ is the unique weak solution of
\begin{equation*}
\arraycolsep=2pt
\left\{
\begin{array}{rcll}
  v_1 -D\Delta v_1 & = & \frac{1-\mathrm{e}^{-\alpha T}}{\alpha T}\chi_{\left(\Omega\textbackslash \mathrm{C}\right)^{\#}} + \chi_{\Omega^{\#}\textbackslash\left(\Omega\textbackslash \mathrm{C}\right)^{\#}}   & \text{ in } \Omega^{\#} , \\
v_1 & = & 1  & \text{ on } \partial{\Omega^{\#}},
\end{array}
\right.
\end{equation*}
and where $\Omega^{\#}\textbackslash\left(\Omega\textbackslash \mathrm{C}\right)^{\#}=\Omega^{\#}\textbackslash \mathrm{B}_{n}(0,(\frac{|\Omega|-L}{\omega_n})^{\frac{1}{n}})=:\mathrm{A}$ since $|\Omega\textbackslash\mathrm{C}|=|\Omega|-L$. Moreover, 
$$
\frac{1-\mathrm{e}^{-\alpha T}}{\alpha T}\chi_{\left(\Omega\textbackslash \mathrm{C}\right)^{\#}} + \chi_{\Omega^{\#}\textbackslash\left(\Omega\textbackslash \mathrm{C}\right)^{\#}}=\frac{1}{T}\int_{0}^T \mathrm{e}^{-\alpha(1-\chi_{\mathrm{A}})t}\mathrm{d}t,
$$
\textit{a.e.} in $\Omega^{\#}$.
Thus, $v_1=\int_0 ^T p^{\Omega^{\#}} _{\chi_{\mathrm{A}}} / T \in \HH^1(\Omega^{\#})$ which concludes the proof.
\end{proof}

\begin{myRem}\normalfont
Note that, if we consider a homogeneous Dirichlet boundary condition in the protection problem~\eqref{modele}, one can determine similarly to Proposition~\ref{worstfield}, on which domain and for which control the cost $\mathcal{J}$ is the highest, i.e.,
$$
\max\limits_{ \substack{ u\in \mathcal{U}_{ad} \\  } } \; \mathcal{J}(u) \leq \int_{0}^T  \int_{\Omega^{\#}} p^{\Omega^{\#}}_{\chi_{\mathrm{B}_{n}\left(0,\left(\frac{L}{\omega_n}\right)^{\frac{1}{n}}\right)}}\mathrm{d}x \mathrm{d}t,
$$
where $p^{\Omega^{\#}}_{\chi_{\mathrm{B}_{n}\left(0,\left(\frac{L}{\omega_n}\right)^{\frac{1}{n}}\right)}}\in\LL^2(]0,T[,\HH^1 _0(\Omega^{\#}))$ is the unique weak solution of the protection problem~\eqref{modele} defined on~$\Omega^{\#}\times]0,T[$ with a homogeneous Dirichlet boundary condition and for the control $u=\chi_{\mathrm{B}_{n}\left(0,\left(\frac{L}{\omega_n}\right)^{\frac{1}{n}}\right)}$ (see~\cite{VOAS} for a similar result on an optimal control problem where the control acts on the divergence operator).
\end{myRem}

\begin{myRem}\normalfont
If one does not restrict Proposition~\ref{worstfield} to bang-bang controls, then one can prove in the same way that,
$$
\min\limits_{ \substack{ u\in \mathcal{U}_{ad} \\  } } \; \mathcal{J}(u)= \mathcal{J}(0)\geq \int_{0}^T \int_{\Omega^ {\#}} p^{\Omega^{\#}}_{0} (x,t)\mathrm{d}x \mathrm{d}t,
$$
where $p^{\Omega^{\#}}_{0}$ is the unique weak solution of the protection problem~\eqref{modele} defined on $\Omega^{\#}\times]0,T[$ for the null control~$u=0$ \textit{a.e.} in $\Omega$.
\end{myRem}

We conclude this section with the following remark.

\begin{myRem}\normalfont
By integrating on~$\Omega\times]0,T[$ the protection problem~\eqref{modele} and using Green's formula, we obtain
$$
\int_0 ^T \int_{\Omega} p_{u} (x,t) \mathrm{d}x \mathrm{d}t= \int_0 ^T \int_{\Omega} \mathrm{e}^{-\alpha\left(1-u(x)\right)t}\mathrm{d}x \mathrm{d}t+D\int_0 ^T \int_{\partial\Omega} \partial_{\nn} p_u (x,t)\mathrm{d}x \mathrm{d}t,
$$
for all $u\in\mathrm{Extr}(\mathcal{V}_{ad})$. Since 
$$
\int_0 ^T \int_{\Omega} \mathrm{e}^{-\alpha\left(1-u(x)\right)t} \mathrm{d}x \mathrm{d}t=T\int_{\{u=1\}}\mathrm{d}x+\int_{0}^T \int_{\{u=0\}}\mathrm{e}^{-\alpha t} \mathrm{d}x \mathrm{d}t=TL+\left(|\Omega|-L\right)\frac{1-\mathrm{e}^{-\alpha T}}{\alpha},
$$
for all $u\in\mathrm{Extr}(\mathcal{V}_{ad})$, then the optimal control problem~\eqref{optimalcontrol} is equivalent to maximizing the functional $\int_0 ^T \int_{\partial\Omega} \partial_{\nn} p_u (x,t)\mathrm{d}x \mathrm{d}t$ on $\mathrm{Extr}(\mathcal{V}_{ad})$. In particular, if the Dirichlet boundary condition is replaced by a Neumann boundary condition in the protection problem~\eqref{modele}, then all bang-bang controls are solutions of the optimal control problem~\eqref{optimalcontrol}. Indeed, if $\partial_{\nn} p _u =g$ on $\partial\Omega\times]0,T[$, where $g\in\LL^2(]0,T[,\HH^{-1/2}(\Omega))$, then we can still prove that $\mathcal{J}$ attains its maximum on the set of bang-bang controls and, moreover,
$$
\int_0 ^T \int_{\Omega} p_{u}(x,t) \mathrm{d}x \mathrm{d}t = D, \qquad \forall u\in\mathrm{Extr}(\mathcal{V}_{ad}),
$$
where $D>0$ is a constant independant of $u\in\mathrm{Extr}(\mathcal{V}_{ad})$.
\end{myRem}

\section{Numerical simulations}\label{num}

In this section we numerically solve an example of the optimal control problem~\eqref{optimalcontrol} in the two-dimensional case $n=2$, in order to illustrate Theorem~\ref{maintheo}. The numerical simulations have been performed using Freefem++ software~\cite{HECHT} with P1-finite elements and standard affine mesh.

\subsection{Numerical methodology}

Starting from an initial data $u_0\in\mathcal{U}_{ad}$,
note that Proposition~\ref{differentielle} allows to obtain an ascent direction~$h_d (u_0)$ of the cost functional~$\mathcal{J}$ at~$u_0$, given by~$h_d (u_0):=qw\circ u_0\in\LL^{\infty}(\Omega)$, since it satisfies~$\mathrm{d}\mathcal{J}(u_0)(h_d (u_0))=\|qw\circ u_0
\|^2_{\LL^2(\Omega)}\geq 0$. To deal with the inequality constraint $\int_{\Omega}u(x)\mathrm{d}x\leq L$, the Uzawa algorithm (see, e.g.,~\cite[Chapter 3 p.64]{ALL}) is used, while the bound constraints are dealt with the projected gradient. In a nutshell, we consider the augmented functional $\mathcal{J}- \lambda_0 F$, where $F : u\in\LL^{\infty}(\Omega)\mapsto \int_\Omega u -L\in\R$ and~$\lambda_0\in\R^+$ is an initial Lagrange multiplier. Since the Fréchet differential of $F$ at $u_0$ is the map~$h\in\LL^{\infty}(\Omega)\mapsto \lambda_0\int_\Omega h \in\R$, one can obtain an ascent direction of the augmented functional at~$u_0$ given by $h_d (u_0)-\lambda_0$. Then the new control is given by
$$
u_1=\mathrm{proj}_{\mathcal{V}}\left(u_0+\eta \left(h_d (u_0)-\lambda_0\right)\right),
$$
where $\eta>0$ is a fixed parameter and $\mathrm{proj}_{\mathcal{V}}$ is the projection operator onto $\mathcal{V}:=\{u\in\LL^{\infty}(\Omega) \mid 0\leq u \leq 1 \text{ \textit{a.e.} in } \Omega \}$ considered in $\LL^2(\Omega)$.
Finally, the Lagrange multiplier is updated as follows
$$
\lambda_{1}:=\lambda_{0}+\mu F(u_1),
$$
where $\mu>0$ is a fixed parameter, and the algorithm restarts with~$u_1$ and~$\lambda_{1}$, and so on. To conclude, note that the algorithm stops when, for all~$i\in~\mathbb{N}^{*}$, the difference between the cost functional $\mathcal{J}$ at the iteration~$20\times i$ and at the iteration~$20\times (i-1)$ is small enough.

\subsection{Numerical results}

For numerical simulations, we consider a diffusion coefficient~$D=0.01$ and a death rate~$\alpha=1$. In the following, the domains considered have the same measure~$|\Omega|=4$ and the maximal surface of the intervention zones is $L=|\Omega|/4=0.4$. 

In Figure~\ref{fig3} (resp. Figure~\ref{fig4}), the numerical simulation is performed on the domain $\Omega_1=]-1,1[\times]0,-1[\cup]0,1[\times]0,2[$ (resp. $\Omega_2=]-3c,3c[\times]-c,c[ \cup ]-c,c[\times]c,3c[ \cup ]-c,c[\times]-3c,-c[ $ for $c=1/\sqrt{5}$). Figure (1a) (resp. (2a)) shows the optimal control, (1b) (resp. (2b)) the optimal protection at the final time~$T$. One observes that the optimal control takes exclusively the two values~$0$ (in orange) and $1$ (in red) in the domain and thus it is a bang-bang control which is in accordance with Theorem~\ref{maintheo}. Furthermore, one sees that the zone of intervention is connected, and in the case of the symmetric domain $\Omega_2$, it is also symmetric. (1c) (resp. (2c)) shows the evolution of the value of $\mathcal{J}$ and (1d) (resp. (2d)) the surface of the intervention zone with respect to the iterations.  We observe that $\mathcal{J}$ seems to converge with some oscillations due to the Lagrange multiplier in order to satisfy the surface constraint.

In Figure~\ref{figg1}, a circular domain is considered $\Omega_3=\mathrm{B}_2 (0,2/\sqrt{\pi})$. As in Figure~\ref{fig3} and Figure~\ref{fig4}, one observes that the optimal control is bang–bang. The zone of intervention is a disk concentrated in the center of the domain with a surface close to $0.4$ as stated in item~\ref{firstitem} of Theorem~\ref{maintheo}. 

In Figure~\ref{figg2}, we consider a rectangular domain $\Omega_4=]-2,2[\times]-\frac{1}{2},\frac{1}{2}[$. As in previous figures, the optimal control is bang–bang. Moreover the zone of intervention is symmetric with respect to the axis $(Ox)$ and $(Oy)$, convex in the $x$-direction and in the $y$-direction and also star-shaped at $0$ as expected from item~\ref{thirditem} of Theorem~\ref{maintheo}.

In Figure~\ref{fig5}, in (5a) are summarized the values of the cost functional at the optimal control on their respective domains. The circular domain and the annular intervention zone described in Proposition~\ref{worstfield} are computed in (5b), with the value of $\mathcal{J}$ who is lower than in the other numerical simulations, as claimed.

\section*{Acknowledgement}
The authors are grateful to Greenshield Company for suggesting this study and to Idriss Mazari-Fouquer for fruitful discussions and valuable advice. The authors were funded by the MATAE project.

\begin{figure}[h!]
    \centering
    \renewcommand\thesubfigure{1\alph{subfigure}}
    \subfloat[]{\includegraphics[scale=0.24, trim = 0.cm 0.1cm 1cm 0.5cm, clip]{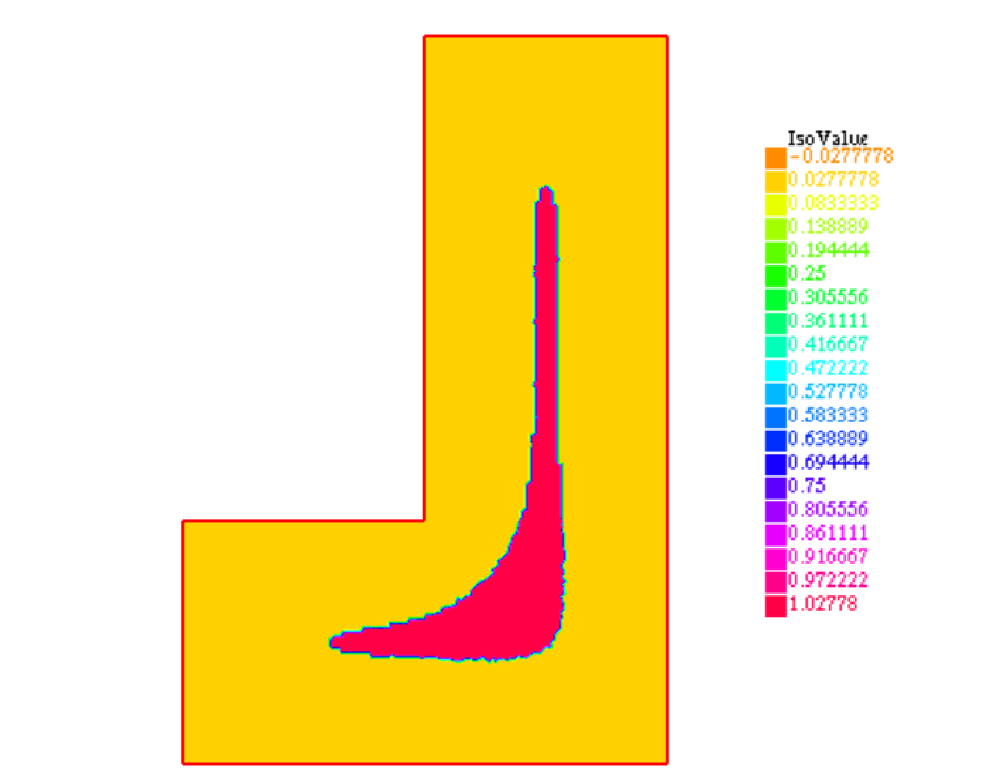}}
    \subfloat[]{\includegraphics[scale=0.24, trim = 0.cm 0.cm 1cm 0.5cm, clip]{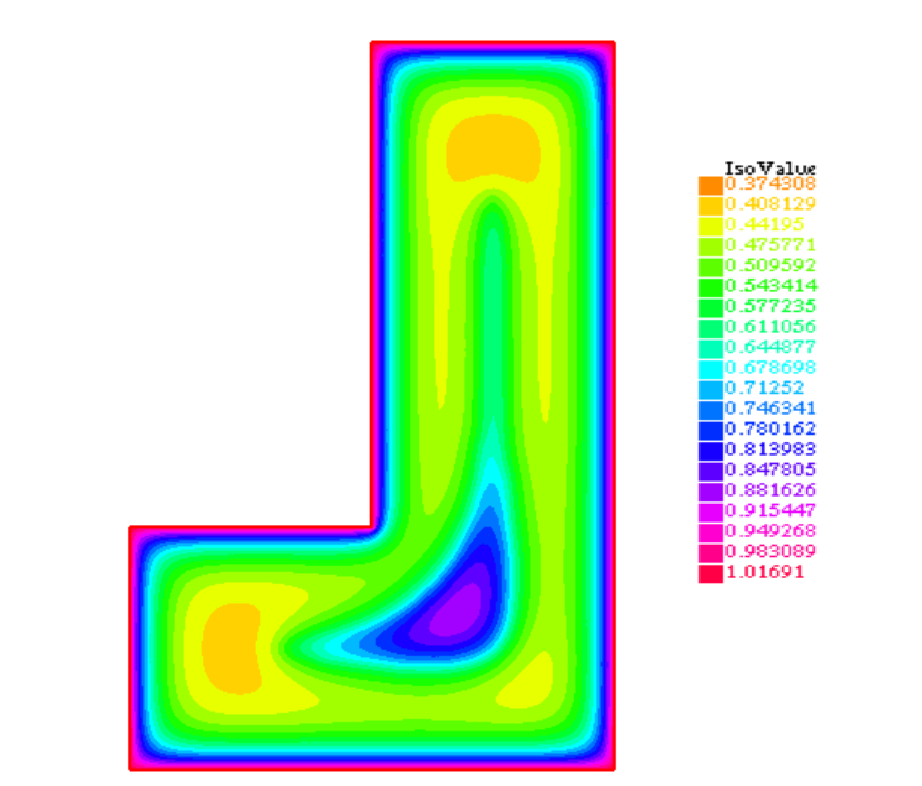}}
    
    \subfloat[]{\includegraphics[scale=0.5]{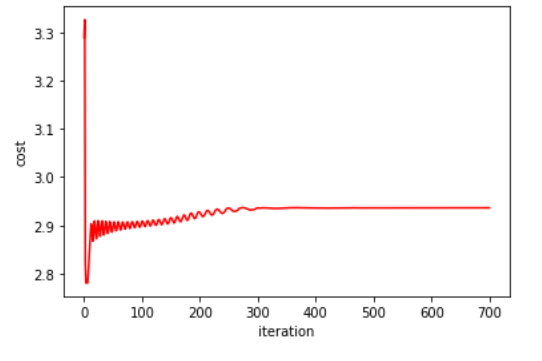}}
    \subfloat[]{\includegraphics[scale=0.5]{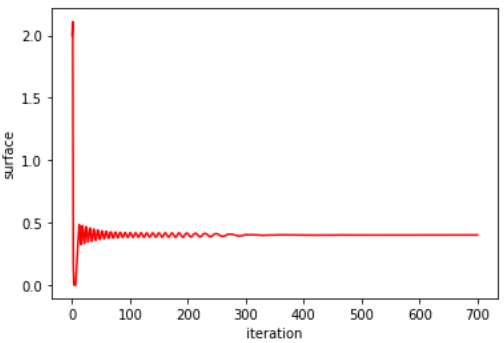}}
    
    \caption{\small{(1a) The optimal control on $\Omega_1$. (1b) The protection for the optimal control at the final time~$T$. The values of the cost functional (1c) and the surface of the intervention zone (1d) with respect to the iterations.}}\label{fig3}
\end{figure}

\begin{figure}[h!]
    \centering
    \renewcommand\thesubfigure{2\alph{subfigure}}
    \subfloat[]{\includegraphics[scale=0.26, trim = 0.cm 0.1cm 1cm 0.5cm, clip]{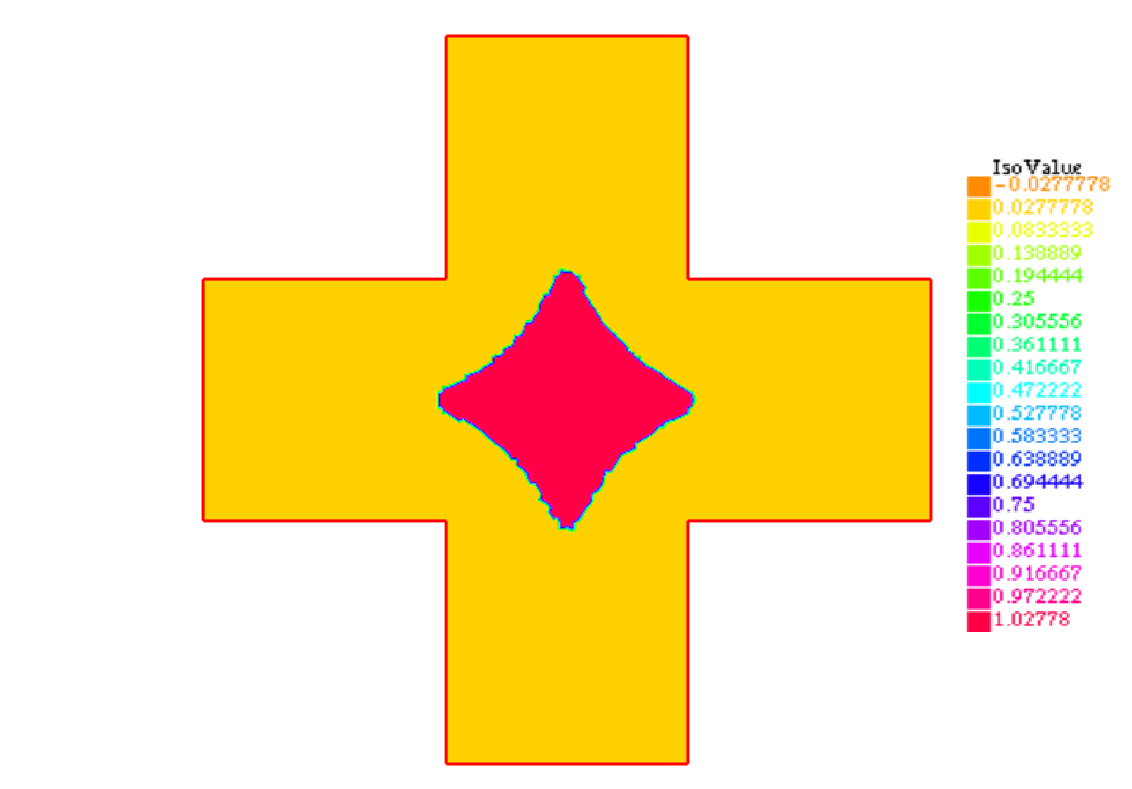}}
    \subfloat[]{\includegraphics[scale=0.26, trim = 0.cm 0.cm 1cm 0.5cm, clip]{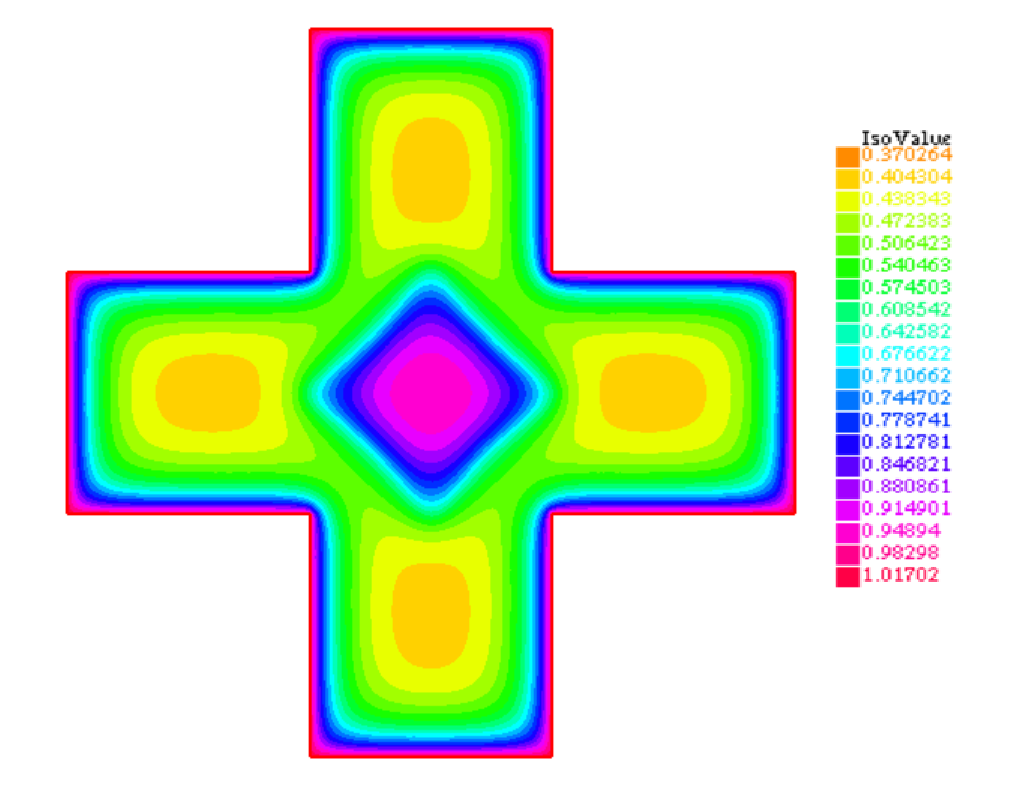}}
    
    \subfloat[]{\includegraphics[scale=0.52]{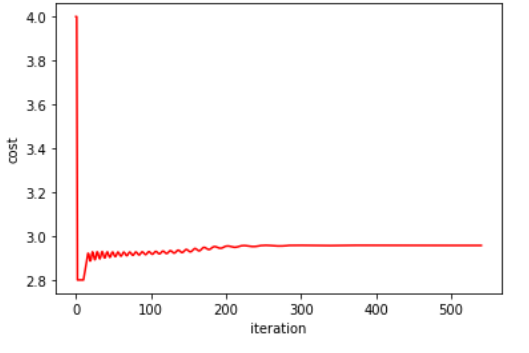}}
    \subfloat[]{\includegraphics[scale=0.52]{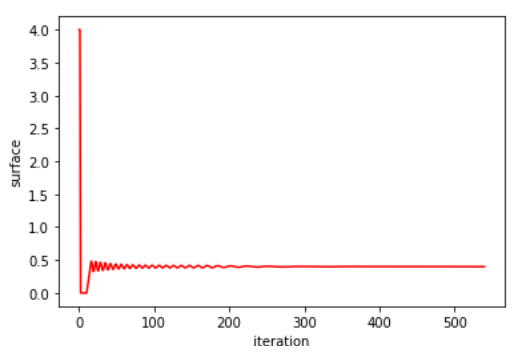}}
    
    \caption{\small{(2a) The optimal control on $\Omega_2$. (2b) The protection for the optimal control at the final time~$T$. The values of the cost functional (2c) and the surface of the intervention zone (2d) with respect to the iterations.}}\label{fig4}
\end{figure}

\begin{figure}[h!]
    \centering
    \renewcommand\thesubfigure{3\alph{subfigure}}
    \subfloat[]{\includegraphics[scale=0.26, trim = 0.cm 0.1cm 1cm 0.5cm, clip]{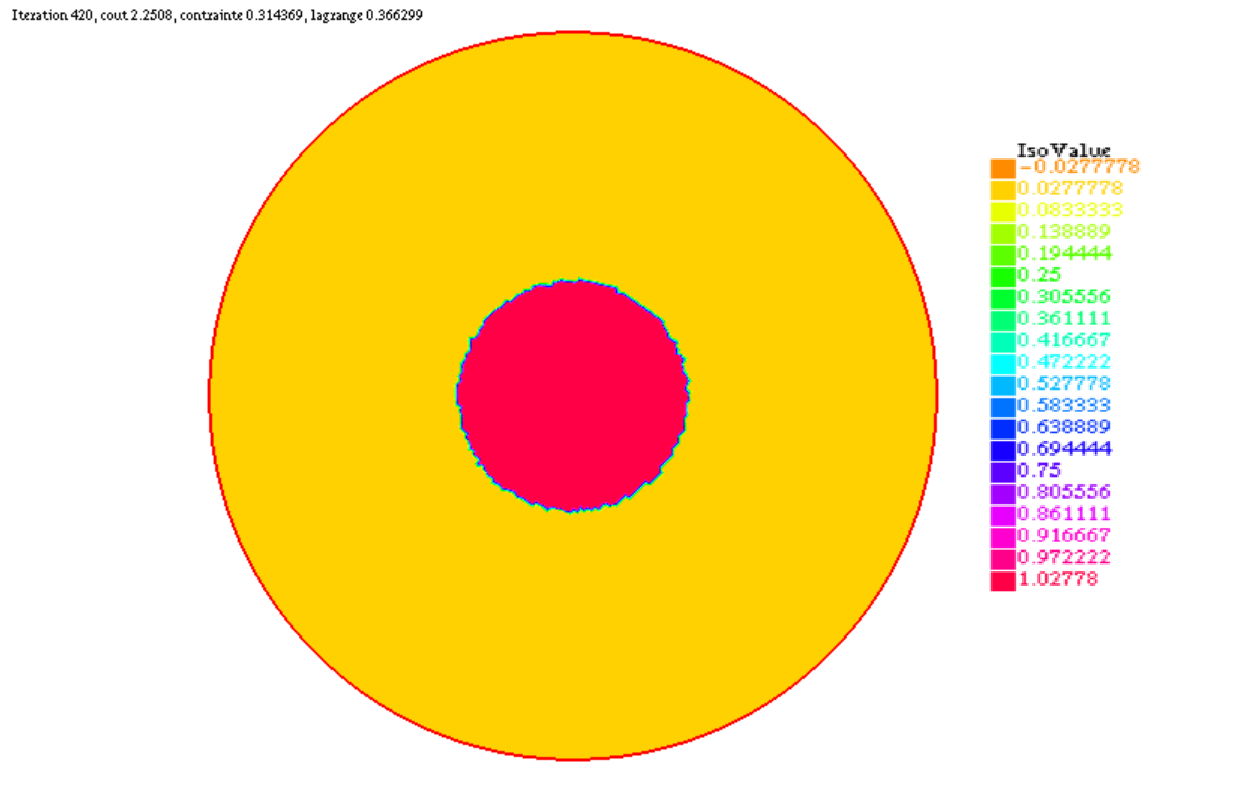}}
    \subfloat[]{\includegraphics[scale=0.26, trim = 0.cm 0.cm 1cm 0.5cm, clip]{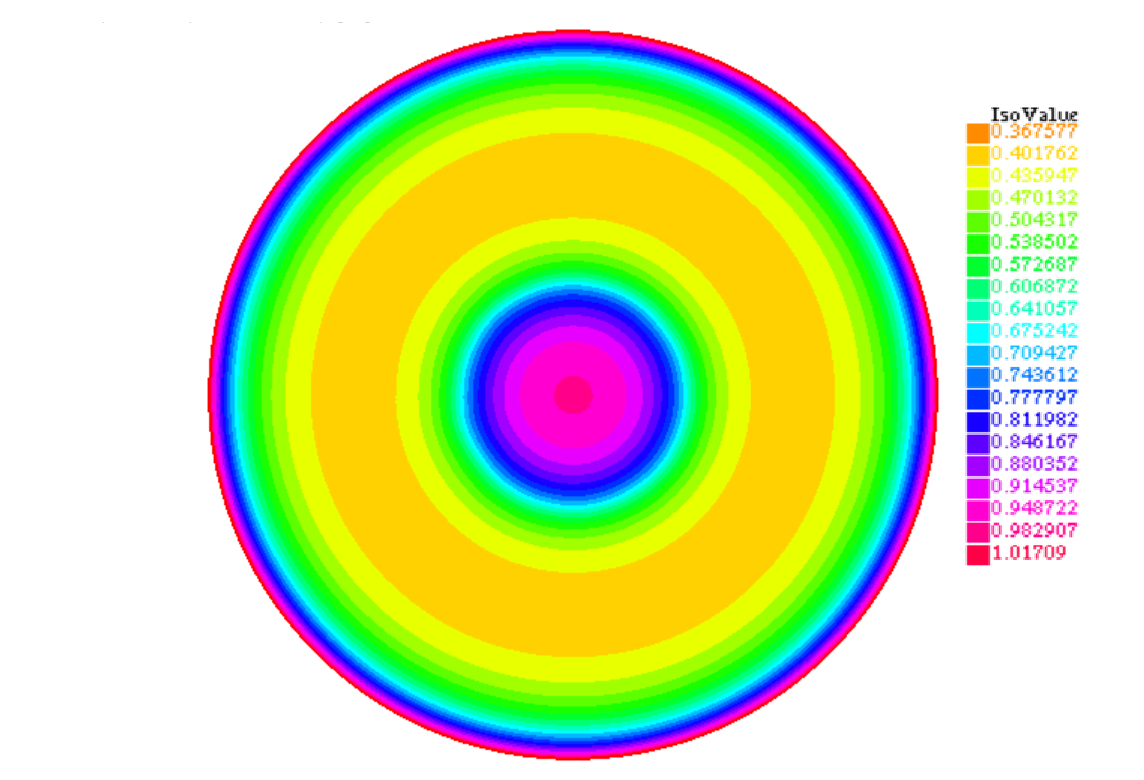}}
    
    \subfloat[]{\includegraphics[scale=0.52]{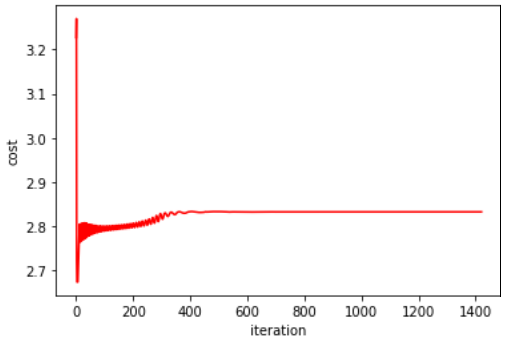}}
    \subfloat[]{\includegraphics[scale=0.52]{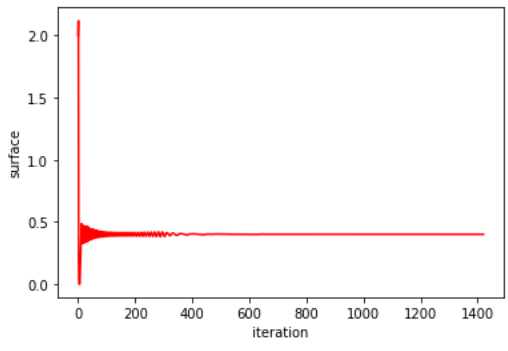}}
    
    \caption{\small{(3a) The optimal control on the circular domain $\Omega_3$. (3b) The protection for the optimal control at the final time $T$. The values of the cost functional (3c) and the surface of the intervention zone (3d) with respect to the iterations.}}\label{figg1}
\end{figure}

\begin{figure}[h!]
    \centering
    \renewcommand\thesubfigure{4\alph{subfigure}}
    \subfloat[]{\includegraphics[scale=0.25, trim = 0.cm 0.1cm 0.cm 0.3cm, clip]{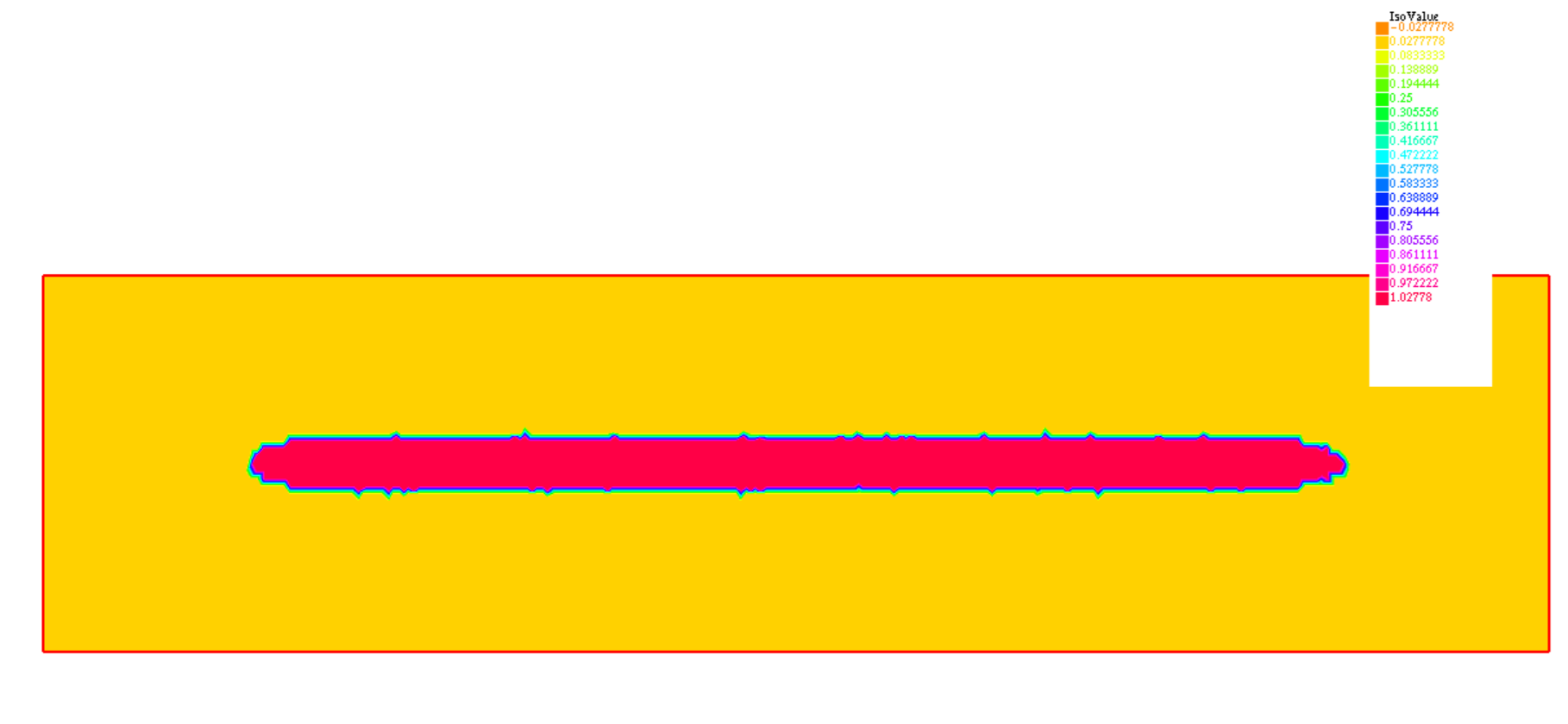}}
     \subfloat[]{\includegraphics[scale=0.25, trim = 0.cm 0.1cm 0cm 0.3cm, clip]{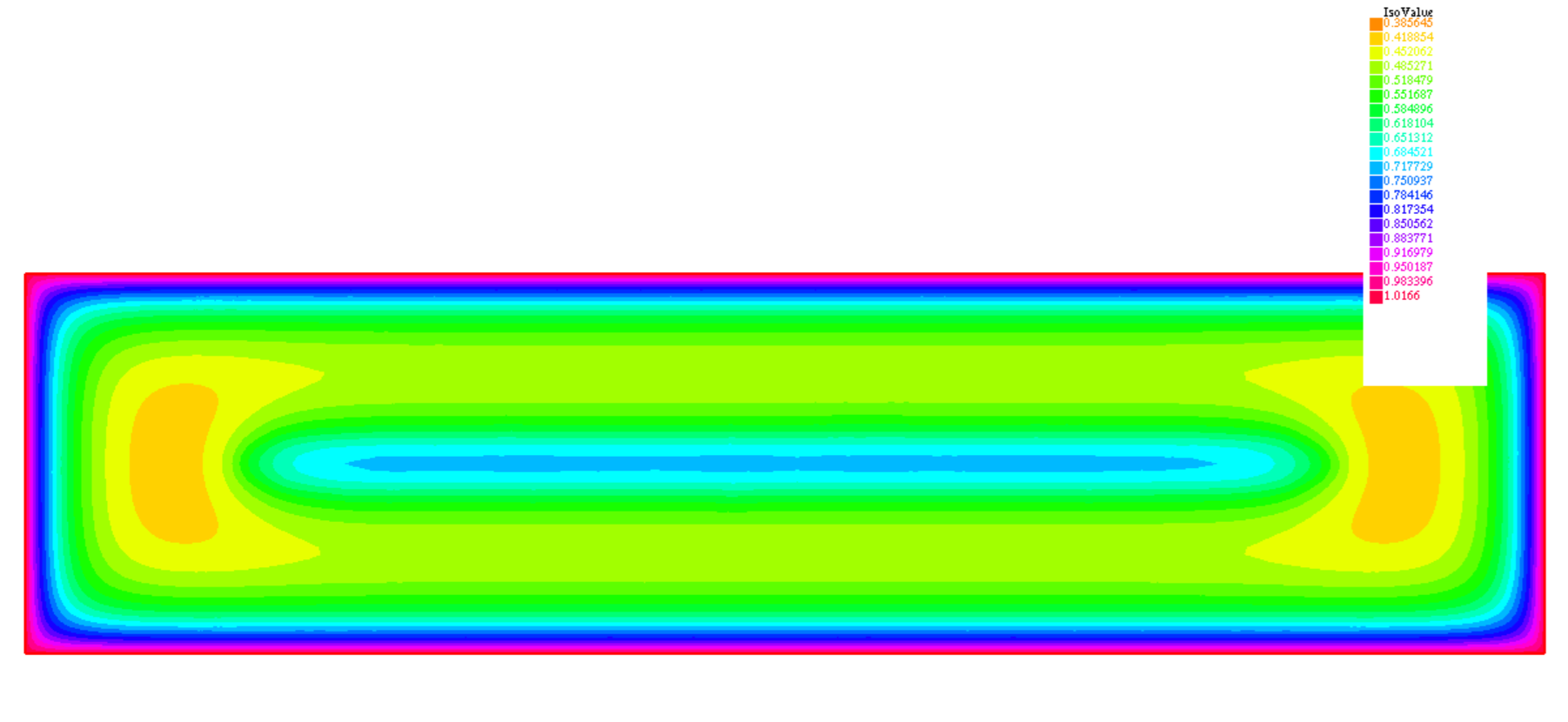}}
 
     \subfloat[]{\includegraphics[scale=0.5]{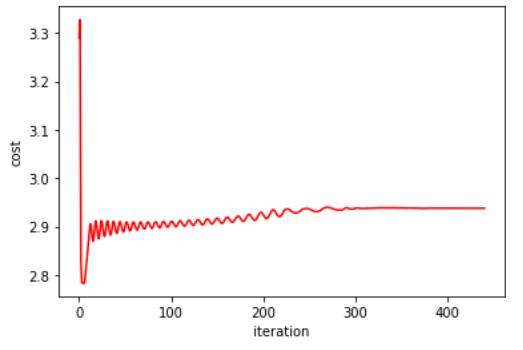}}
     \subfloat[]{\includegraphics[scale=0.5]{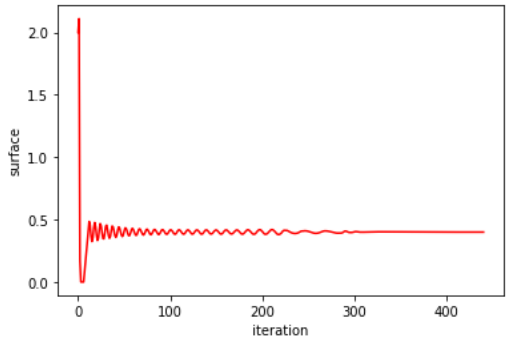}}
    
     \caption{\small{(4a) The optimal control on the rectangular domain $\Omega_4$. (4b) The protection for the optimal control at the final time $T$. The values of the cost functional (4c) and the surface of the intervention zone (4d) with respect to the iterations.}}\label{figg2}
\end{figure}

\begin{figure}[!ht]
    \centering
    \renewcommand\thesubfigure{5\alph{subfigure}}
    \subfloat[]{\begin{tabular}[b]{|l|c|c|c|c|}
    \hline
    The domain  & $\Omega_1$ & $\Omega_2$ & $\Omega_3$ & $ \Omega_4$    \\
    \hline 
    value of $\mathcal{J}$ at the optimal control & 2.93643  & 2.95918 & 2.83285 & 2.93773   \\
    \hline
    Surface of the intervention zone &  0.400416 & 0.399502 & 0.399777 & 0.399541 \\
     \hline
    \end{tabular}}
    
    \subfloat[]{\includegraphics[scale=0.5]{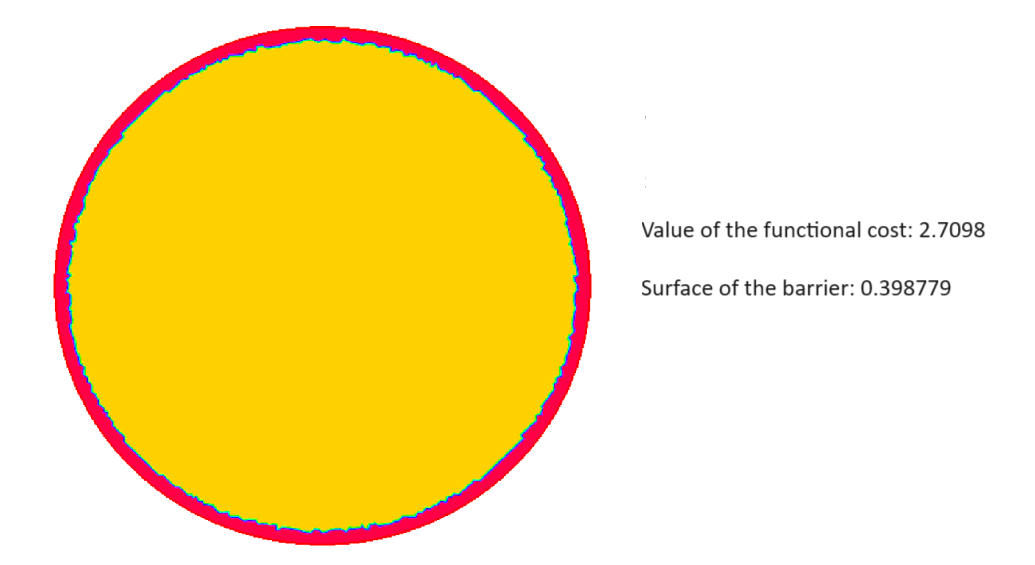}}
    
    \caption{\small{(5a) The value of $\mathcal{J}$ at the optimal control and the surface of the intervention zone on the different domains considered. (5b) A circular domain with an annular intervention zone  and the value of $\mathcal{J}$ for this control.}}
    \label{fig5}
\end{figure}

\appendix

\section{Notions and results on Schwarz and Steiner symmetrizations}\label{remind}

The aim of this section is to recall and to prove some results on the Schwarz and Steiner symmetrizations. For more details we refer to standard references such as~\cite{Baernstein,Burchard,Henrot,KAW,KESAVAN}.

\subsection{Schwarz symmetrization}\label{Schwarz}

\begin{myDefn}[Symmetric rearrangement]
Let~$\mathrm{C}$ be a measurable subset of~$\R^{n}$ such that $|\mathrm{C}|<\infty$. The symmetric rearrangement of $\mathrm{C}$ is the subset $\mathrm{C}^{\#}\subset\R^n$  defined by
$$
\mathrm{C}^{\#}:=\mathrm{B}_n\left(0,\left(\frac{|\mathrm{C}|}{{\omega_n}}\right)^{\frac{1}{n}}\right),
$$
where
$\omega_n:=|\mathrm{B}_n(0,1)|$.
\end{myDefn}

\begin{myDefn}[Schwarz symmetrization]
Let $\mathrm{C}$ be a measurable subset of~$\R^n$ such that $|\mathrm{C}|<\infty$, and let~$f : \mathrm{C}\rightarrow \mathbb{R}$ be a nonnegative measurable function. The Schwarz symmetrization of $f$ is the function~$f^{\#}$ defined by
$$
f^{\#}(x):=\int_{0}^{+\infty} \chi_{{\left\{f>t\right\}}^{\#}}(x)\,\mathrm{d}t,
$$
for all $x\in\mathrm{C}^\#$.
\end{myDefn}

\begin{myExa}\label{indicSch}
Let $\Omega$ be a nonempty bounded open subset $\R^n$ and let $\mathrm{C}$ be a measurable subset of $\Omega$. Consider the map $f : x\in\Omega \mapsto \alpha\chi_{\mathrm{C}}(x)+\beta\chi_{\Omega\backslash\mathrm{C}} (x) \in\R^+$, where~$\alpha,\beta$ are constant such that~$\alpha\geq\beta\geq0$, and $\chi_{\mathrm{C}}$ (resp. $\chi_{\Omega\backslash\mathrm{C}}$) is the characteristic function of $\mathrm{C}$ (resp. of $\Omega\backslash\mathrm{C}$). Then the Schwarz symmetrization of $f$ is $f^\# : x\in\Omega^{\#}\mapsto\alpha\chi_{\mathrm{C}^\#} (x)+\beta\chi_{\Omega^{\#}\backslash\mathrm{C}^{\#}} (x)$.
\end{myExa}

The next comparison theorem for elliptic equations is proved in~\cite{ALVIN,ALVMATA,Chiti,TROMBV} or~\cite[Section 9]{Trombetti}), and it is one of the key points to derive item~\ref{firstitem} of Theorem~\ref{maintheo}. 

\begin{myTheorem}\label{theoAlv}
Let us consider $\Omega$ a nonempty bounded connected open subset of $\R^{n}$, with a lipschitz boundary $\partial{\Omega}$, $f\in\LL^2(\Omega)$ a nonnegative function,~$D>0$ a constant, $Z_f\in\HH^1 _0(\Omega)$ the unique weak solution of the Dirichlet problem,
\begin{equation*}
\arraycolsep=2pt
\left\{
\begin{array}{rcll}
 Z_f -D\Delta Z_f & = & f   & \text{ in } \Omega , \\
Z_f & = & 0  & \text{ on } \partial{\Omega} ,
\end{array}
\right.
\end{equation*}
and $v\in\HH^1 _0(\Omega^{\#})$ the unique weak solution of the Dirichlet problem
\begin{equation*}
\arraycolsep=2pt
\left\{
\begin{array}{rcll}
  v -D\Delta v & = & f^{\#}   & \text{ in } \Omega^{\#} , \\
v & = & 0  & \text{ on } \partial{\Omega^{\#}} ,
\end{array}
\right.
\end{equation*}
where $\Omega^{\#}$ is the symmetric rearrangement of $\Omega$ and $f^{\#}$ is the Schwarz symmetrization of $f$. Then one has
$$
\int_{\Omega}Z_f(x)\mathrm{d}x\leq\int_{\Omega^{\#}} v(x)\mathrm{d}x.
$$
\end{myTheorem}

\subsection{Steiner symmetrization}\label{Steiner}

In this subsection, let us assume that $n\geq2$ and let us denote~$\R^n=\R^{n_1}\times\R^{n_2}$, where~$n_1\in\N^*$,~$n_2\in\N^*$, such that $n_1+n_2=n$. We refer to the variables in $\R^{n_1}$
as~$x=(x_1,...,x_{n_1})$, and to the variables in $\R^{n_2}$ as $y=(y_1,...,y_{n_2})$.

\begin{myDefn}[Steiner symmetrization of sets]
Let $\mathrm{C}\subset\R^{n}$ be a measurable set such that~$|\mathrm{C}|<\infty$. The $(n_1,\cdot)$-Steiner symmetrization of $\mathrm{C}$ is defined by
$$
S_{(n_1,\cdot)}(\mathrm{C}):=\left\{ (x,y)\in\R^{n_1}\times\R^{n_2} \mid x\in \mathrm{C}(y)^{\#},~y\in\R^{n_2}  \right\},
$$
where, for all~$y\in\R^{n_2}$, $\mathrm{C}(y)=\{x\in\R^{n_1}\mid (x,y)\in \mathrm{C}\}$ is the~$y$-slice of $\mathrm{C}$, and $\mathrm{C}(y)^{\#}$ is the symmetric rearrangement of $\mathrm{C}(y)$. One defines the $(\cdot,n_2)$-Steiner symmetrization of $\mathrm{C}$ as~$S_{(\cdot,n_2)}(\mathrm{C}):=\Phi^{-1}\left(S_{(n_2,\cdot)}(\Phi(\mathrm{C}))\right)$
where $\Phi:(x,y)\in \R^{n_1}\times\R^{n_2}\mapsto (y,x)\in\R^{n_2}\times\R^{n_1}$.
\end{myDefn}

The next example is involved in item~\ref{seconditem} and item~\ref{thirditem} of Theorem~\ref{maintheo}.

\begin{myExa}\label{steinercart}
Let $\mathrm{C}=\mathrm{C_{n_1}}\times\mathrm{C}_{n_2}\subset\R^n$ such that $\mathrm{C}_{n_1}\subset\R^{n_1}$ (resp. $\mathrm{C}_{n_2} \subset\R^{n_2})$ is a measurable set of $\R^{n_1}$ (resp. of~$\R^{n_2}$) with a finite measure.
Then the~$(n_1,\cdot)$-Steiner symmetrization of $\mathrm{C}$ is given by
$S_{(n_1,\cdot)}(\mathrm{C})=\mathrm{C}_{n_1} ^{\#}\times \mathrm{C}_{n_2}$, where $\mathrm{C}_{n_1} ^{\#}$ is the symmetric rearrangement of $\mathrm{C}_{n_1}$.
\end{myExa}

From the Steiner symmetrization definition, one deduces the following proposition.
 
\begin{myProp}\label{propstei}
Let $\mathrm{C}\subset\R^{n}$ be a measurable set such that~$|\mathrm{C}|<\infty$. Then:
\begin{enumerate}[label=\arabic*)]
    \item $|S_{(n_1,\cdot)}(\mathrm{C})|=|S_{(\cdot,n_2)}(\mathrm{C})|=|\mathrm{C}|$;
    \item $S_{(n_1,\cdot)}(\mathrm{C})$ (resp. $S_{(\cdot,n_2)}(\mathrm{C})$) is symmetric with respect to the hyperplane $x_i=0$ for all~$i\in\{1,...,n_1\}$ (resp. with respect to the hyperplane $y_j=0$ for all $j\in\{1,...,n_2\}$);
    \item $S_{(n_1,\cdot)}(\mathrm{C})$ (resp. $S_{(\cdot,n_2)}(\mathrm{C})$) is convex in the $x_i$-direction for all $i\in\{1,...,n_1\}$ (resp. in the $y_j$-direction for all $j\in\{1,...,n_2\}$).
\end{enumerate}
\end{myProp}

\begin{myProp}\label{properties}
Let $\mathrm{C}\subset\R^{n}$ be a measurable set such that $|\mathrm{C}|<\infty$. Then the $(\cdot,n_2)$-Steiner symmetrization of $S_{(n_1,\cdot)}(\mathrm{C})$, i.e. $S_{(\cdot,n_2)}(S_{(n_1,\cdot)}(\mathrm{C}))$, is: 
\begin{enumerate}[label=\Alph*)]
    \item symmetric in the hyperplane $x_i=0$ for all $i\in\{1,...,n_1\}$ and in the hyperplane~$y_j=0$ for all~$j\in\{1,...,n_2\}$;
    \item convex in the $x_i$-direction for all $i\in\{1,...,n_1\}$ and convex in the $y_j$-direction for all~$j\in\{1,...,n_2\}$;
    \item star-shaped at $0\in S_{(\cdot,n_2)}(S_{(n_1,\cdot)}(\mathrm{C}))$.
\end{enumerate}
\end{myProp}

\begin{proof}
$A)$ can be easily proved using item~$2)$ of Proposition~\ref{propstei}. \\
$B)$ By item~$3)$ of Proposition~\ref{propstei}, the set $S_{(\cdot,n_2)}\left(S_{(n_1,\cdot)}(\mathrm{C})\right)$ is convex in the $y_j$-direction for all~$j\in\{1,...,n_2\}$.  Let $i\in\{1,...,n_1\}$ be fixed and consider $(x',y), (x'',y)\in S_{(\cdot,n_2)}\left(S_{(n_1,\cdot)}(\mathrm{C})\right)$ where $x'=\left(x_{k} '\right)_{k\in\{1,...,n_1\}}$ and $x''=\left(x_{k}''\right)_{k\in\{1,...,n_1\}}$,  such that $x_{k}'=x_{k} ''$, for all $k\in\{1,...,n_1\}\backslash\{ i \}$. Then $y\in S_{(n_1,\cdot)}(\mathrm{C})(x')^{\#}$ and $y\in S_{(n_1,\cdot)}(\mathrm{C})(x'')^{\#}$. Without loss of generality, we can assume that~$|x_i''|\leq |x_i'|$, then since~$S_{(n_1,\cdot)}(\mathrm{C})$ is symmetric with respect to the hyperplane $x_i=0$ and convex in the $x_i$-direction, then 
$$
S_{(n_1,\cdot)}(\mathrm{C})(x')\subset S_{(n_1,\cdot)}(\mathrm{C})(x_1',...,x_{i-1}',z_i,x_{i+1}',..., x_n'),
$$
 for all $z_i \in [-x_i ',x_i ']$. In particular, since $|x_i''|\leq |x_i'|$, then~$x_i''\in[-x_i ', x_i']$ and $[x_i'',x_i']\subset[-x_i ',x_i']$. It follows that 
 $$
 S_{(n_1,\cdot)}(\mathrm{C})(x') \subset S_{(n_1,\cdot)}(\mathrm{C})(x_1',...,x_{i-1}',z_i,x_{i+1}',...,x_n'),$$ for all $z_i \in [x_1 '',x_1 ']
 $, thus 
 $$
 S_{(n_1,\cdot)}(\mathrm{C})(x') ^{\#} \subset S_{(n_1,\cdot)}(\mathrm{C})(x_1',...,x_{i-1}',tx_i'+(1-t)x_i'', x_{i+1}',..., x_n') ^{\#},
 $$
 and then $y\in S_{(n_1,\cdot)}(\mathrm{C})(x_1',...,x_{i-1}',tx_i'+(1-t)x_i'',x_{i+1}',..., x_n') ^{\#}$ for all $t\in[0,1]$. One gets that
 $$
 (x_1',...,x_{i-1}',tx_i'+(1-t)x_i'',x_{i+1}',..., x_n',y)=t(x',y)+(1-t)(x'',y)\in S_{(\cdot,n_2)}\left(S_{(n_1,\cdot)}(\mathrm{C})\right),
 $$
 therefore $S_{(\cdot,n_2)}\left(S_{(n_1,\cdot)}(\mathrm{C})\right)$ is convex in the $x_i$-direction.\\
$C)$ Let $(x',y')\in S_{(\cdot,n_2)}(S_{(n_1,\cdot)}(\mathrm{C}))$. Since $S_{(\cdot,n_2)}(S_{(n_1,\cdot)}(\mathrm{C}))$ is symmetric with respect to the hyperplanes $x_i=0$ and convex in the $x_i$-direction, for all $i\in\{1,...,n_1\}$, then one gets $(0_{n_1},y')\in S_{(\cdot,n_2)}(S_{(n_1,\cdot)}(\mathrm{C}))$, where $0_{n_1}$ is the null vector of $\R^{n_1}$. In the same way, using the symmetry with respect to the hyperplane $y_j=0$ and the convexity of $S_{(\cdot,n_2)}(S_{(n_1,\cdot)}(\mathrm{C}))$ in the~$y_j$-direction for all $j\in\{1,...,n_2\}$, it follows that $0=(0_{n_1},0_{n_2})\in S_{(\cdot,n_2)}(S_{(n_1,\cdot)}(\mathrm{C}))$. Now, let us consider~$z\in[0,(x',y')]$. Then, there exists~$t\in[0,1]$ such that $z=t(x',y')$. Using the symmetry with respect to the hyperplanes $x_i=0$ and the convexity in the $x_i$-direction,  for all $i\in\{1,...,n_1\}$, one gets that~$(tx',y')\in S_{(\cdot,n_2)}(S_{(n_1,\cdot)}(\mathrm{C}))$. Then, in the same way one gets $(tx',ty')\in S_{(\cdot,n_2)}(S_{(n_1,\cdot)}(\mathrm{C}))$ and it follows that~$z\in S_{(\cdot,n_2)}(S_{(n_1,\cdot)}(\mathrm{C}))$, thus $[0,(x',y')]\subset S_{(\cdot,n_2)}(S_{(n_1,\cdot)}(\mathrm{C}))$ which concludes the proof.
\end{proof}

\begin{myDefn}[Steiner symmetrization of functions]
Let $\mathrm{C}$ be a measurable subset of $\R^n$ such that~$|\mathrm{C}|<\infty$, and let~$f : \mathrm{C}\rightarrow \mathbb{R}$ be a nonnegative measurable function. The $(n_1,\cdot)$-Steiner symmetrization of $f$ is the function $S_{(n_1,\cdot)}(f)$ defined by
\begin{equation*}
S_{(n_1,\cdot)}(f)(x,y):=\int_{0}^{+\infty} \chi_{S_{(n_1,\cdot)}(\left\{f>t\right\})}(x,y)\mathrm{d}t,
\end{equation*}
for all $(x,y)\in S_{(n_1,\cdot)}(\mathrm{C})$.
\end{myDefn}

\begin{myExa}\label{indicSteiner}
Let $\Omega$ be a nonempty bounded open subset $\R^n$ and let $\mathrm{C}$ be a measurable subset of $\Omega$. Consider the map $f : (x,y)\in\Omega \mapsto \alpha\chi_{\mathrm{C}}(x,y)+\beta\chi_{\Omega\backslash\mathrm{C}} (x,y) \in\R^+$, where~$\alpha,\beta$ are constant such that $\alpha\geq\beta\geq0$, and $\chi_{\mathrm{C}}$ (resp. $\chi_{\Omega\backslash\mathrm{C}}$) is the characteristic function of $\mathrm{C}$ (resp. of $\Omega\backslash\mathrm{C}$). Then the~$(n_1,\cdot)$-Steiner symmetrization of $f$ is $S_{(n_1,\cdot)}(f) :(x,y)\in S_{(n_1,\cdot)}\left(\Omega\right) \mapsto \alpha\chi_{S_{(n_1,\cdot)}\left(\mathrm{C}\right)} (x,y) +\beta\chi_{S_{(n_1,\cdot)}\left(\Omega\right)\backslash S_{(n_1,\cdot)}(\mathrm{C})} (x,y) \in\R$.
\end{myExa}

The next Steiner symmetrization comparison theorem  for elliptic equations (see~\cite{CHIA}) is the main theorem needed to prove item~\ref{seconditem} and item~\ref{thirditem} of Theorem~\ref{maintheo}.

\begin{myTheorem}\label{theoSt}
Let $\Omega\subset\R^{n}$ and assume that $\Omega=\Omega_{n_1} \times\Omega_{n_2}$, where $\Omega_{n_1}\subset\R^{n_1}$ (resp.~$\Omega_{n_2}\subset\R^{n_2}$) is a nonempty bounded open subset of $\R^{n_1}$ (resp. of $\R^{n_2}$). Let $f\in\LL^{p}(\Omega)$ be a nonnegative function with $p>n/2$,~$D>0$, and consider~$Z_f\in\HH^1 _0 (\Omega)$ the unique weak solution of the Dirichlet problem
\begin{equation*}
\arraycolsep=2pt
\left\{
\begin{array}{rcll}
 Z_f -D\Delta Z_f & = & f   & \text{ in } \Omega , \\
Z_f & = & 0  & \text{ on } \partial{\Omega} ,
\end{array}
\right.
\end{equation*}
and $v\in\HH^1 _0 (S_{(n_1,\cdot)}(\Omega))$ the unique weak solution of the Dirichlet problem
\begin{equation*}
\arraycolsep=2pt
\left\{
\begin{array}{rcll}
  v -D\Delta v & = & S_{(n_1,\cdot)}(f)  & \text{ in } S_{(n_1,\cdot)}(\Omega) , \\
v & = & 0  & \text{ on } \partial{S_{(n_1,\cdot)}(\Omega)},
\end{array}
\right.
\end{equation*}
where $S_{(n_1,\cdot)}(\Omega)=\Omega_{n_1} ^{\#}\times\Omega_{n_2}$. Then, one has
$$
\int_{\Omega_{n_1}} Z_f(x,y)\mathrm{d}x\leq\int_{\Omega_{n_1} ^{\#}} v(x,y)\mathrm{d}x,
$$
for almost all $y\in\Omega_{n_2}$.
\end{myTheorem}

\section{Bang-bang controls and extreme points}\label{appendixB}

In this section, we remind the proof that extreme points of $\mathcal{V}_{ad}$ are the set of bang-bang controls.

\begin{myProp}\label{ProofExtrBang}
Let $\mathcal{V}_{ad}$ be the set defined in~\eqref{setconstr}.
Then, 
$$
\mathrm{Extr}(\mathcal{V}_{ad})=\left\{ u\in\LL^{\infty}(\Omega) \mid u\in\left\{0,1\right\} \text{ \textit{a.e.} in } \Omega, \text{ and } \int_{\Omega}u(x)\mathrm{d}x= L\right\}.
$$
\end{myProp}

\begin{proof}
$\supseteq$ Let $u$ be a bang-bang control, then there exists $\mathrm{C}\subset\Omega$, $|\mathrm{C}|=L$, such that $u=\chi_\mathrm{C}$. Assume that $u\notin\mathrm{Extr}(\mathcal{V}_{ad})$, then there exists $v_1,v_2\in\mathcal{V}_{ad}$ such that $u=(v_1+v_2)/2$, where $v_1\neq v_2$. If $x\in \mathrm{C}$, then $v_1(x)+v_2(x)=2$, thus $v_1(x)=v_2(x)=1$, and if $x\notin \mathrm{C}$, then $v_1(x)+v_2(x)=0$, thus $v_1(x)=v_2(x)=0$. It follows that $v_1=v_2=u$ which is a contradiction. \\
$\subseteq$ Let $u\in\mathrm{Extr}(\mathcal{V}_{ad})$ and assume that $u$ is not a bang-bang control. Then $\tilde{\Omega}:=\{x\in\Omega, 0<u(x)<1\}$ satisfies $|\tilde{\Omega}|>0$. Therefore, there exists~$h\in\LL^{\infty}(\tilde{\Omega})$, such that $0<h<1$ \textit{a.e.} on~$\tilde{\Omega}$ and $0<u(x)-h(x)<u(x)<u(x)+h(x)<1$ for almost all $x\in\tilde{\Omega}$. Moreover, the map~$f : r\in\R^+ \to |\mathrm{B}_n (0,r)\cap\tilde{\Omega}|\in \R^+$ is continuous with $f(0)=0$ and $f(r)=|\tilde{\Omega}|$ for all $r\geq R$, if~$R>0$ is large enough. Thus, there exists $\mathrm{U}\subset\tilde{\Omega}$ such that~$|\mathrm{U}|=|\tilde{\Omega}|/2$. Let us consider
$$
\fonction{g }{\tilde{\Omega}}{\R}{x}{ g(x) := \left\{
\begin{array}{lcll}
\epsilon	&   & \text{ if } x\in \mathrm{U}, \\
\displaystyle-\epsilon\frac{\int_{\mathrm{U}} h}{\int_{\tilde{\Omega}\backslash \mathrm{U}} h}	&   & \text{ if } x\in \tilde{\Omega}\backslash \mathrm{U},
\end{array}
\right.
}
$$
where $\epsilon\in]0,1[\cap]0,\frac{{\int_{\tilde{\Omega}\backslash \mathrm{U}} h}}{\int_{\mathrm{U}} h}[$, and  
$$
y:=\left\{\begin{array}{lcll}
u-gh	&   & \text{ if } x\in \tilde{\Omega} \\
u	&   & \text{ if } x\in \Omega\backslash \tilde{\Omega},
\end{array}
\right. \text{ and }
z:=\left\{\begin{array}{lcll}
u+gh	&   & \text{ if } x\in \tilde{\Omega}, \\
u	&   & \text{ if } x\in \Omega\backslash \tilde{\Omega}.
\end{array}
\right.
$$
Then $0\leq y\leq1$ and $0\leq z\leq1$, \textit{a.e.} on $\Omega$, and,
$$
\int_{\Omega}y=\int_{\Omega\backslash \tilde{\Omega}}u+\int_{\tilde{\Omega}}u -\epsilon\int_{\mathrm{U}}h+\epsilon\frac{\int_{\mathrm{U}} h}{\int_{\tilde{\Omega}\backslash \mathrm{U}} h}\int_{\tilde{\Omega}\backslash \mathrm{U}}h=\int_{\Omega} u=L,
$$
and the same calculation applies to $z$. Thus~$y,z\in\mathcal{V}_{ad}$, and $u=(y+z)/2$, with $y\neq z$, thus $u$ is not an extreme point of $\mathcal{V}_{ad}$ which is a contradiction.
\end{proof}

\bibliographystyle{abbrv}
\bibliography{biblio}

\begin{thebibliography}{10}

\bibitem{ALL}
G.~Allaire.
\newblock {\em Conception optimale de structures}.
\newblock Mathématiques et Applications. Springer-Verlag Berlin Heidelberg, 2007.

\bibitem{WOLB2}
L.~Almeida, Y.~Privat, M.~Strugarek, and N.~Vauchelet.
\newblock Optimal releases for population replacement strategies: Application to wolbachia.
\newblock {\em SIAM Journal on Mathematical Analysis}, 51(4):3170--3194, 2019.

\bibitem{ALVNIT}
A.~Alvino, C.~Nitsch, and C.~Trombetti.
\newblock A {T}alenti comparison result for solutions to elliptic problems with {R}obin boundary conditions.
\newblock {\em Comm. Pure Appl. Math.}, 76(3):585--603, 2023.

\bibitem{LIONS2}
A.~Alvino, G.~Trombetti, J.~I. Diaz, and P.~L. Lions.
\newblock Elliptic equations and {S}teiner symmetrization.
\newblock {\em Comm. Pure Appl. Math.}, 49(3):217--236, 1996.

\bibitem{ALVIN}
A.~Alvino, G.~Trombetti, and P.-L. Lions.
\newblock Comparison results for elliptic and parabolic equations via {S}chwarz symmetrization.
\newblock {\em Ann. Inst. H. Poincar\'{e} C Anal. Non Lin\'{e}aire}, 7(2):37--65, 1990.

\bibitem{ALVMATA}
A.~Alvino, G.~Trombetti, P.-L. Lions, and S.~Matarasso.
\newblock Comparison results for solutions of elliptic problems via symmetrization.
\newblock {\em Annales de l'Institut Henri Poincaré C, Analyse non linéaire}, 16(2):167--188, 1999.

\bibitem{AMATO}
V.~Amato, A.~Gentile, and A.~L. Masiello.
\newblock Comparison results for solutions to {$p$}-{L}aplace equations with {R}obin boundary conditions.
\newblock {\em Ann. Mat. Pura Appl. (4)}, 201(3):1189--1212, 2022.

\bibitem{Baernstein}
A.~Baernstein, II.
\newblock A unified approach to symmetrization.
\newblock In {\em Partial differential equations of elliptic type ({C}ortona, 1992)}, volume XXXV of {\em Sympos. Math.}, pages 47--91. Cambridge Univ. Press, Cambridge, 1994.

\bibitem{BANDLE}
C.~Bandle.
\newblock On symmetrizations in parabolic equations.
\newblock {\em J. Analyse Math.}, 30:98--112, 1976.

\bibitem{BANDLE2}
C.~Bandle.
\newblock {\em Isoperimetric inequalities and applications}, volume~7 of {\em Monographs and Studies in Mathematics}.
\newblock Pitman (Advanced Publishing Program), Boston, Mass.-London, 1980.

\bibitem{BEHNCKE}
H.~Behncke.
\newblock Optimal control of deterministic epidemics.
\newblock {\em Optim. Control Appl. Methods}, 21(6):269--285, 2000.

\bibitem{Burchard}
A.~Burchard.
\newblock A short course on rearrangement inequalities.
\newblock {\em Lecture notes, IMDEA Winter School, Madrid}, 2009.

\bibitem{BURT}
G.~Burton.
\newblock Rearrangements of functions, maximization of convex functionals, and vortex rings.
\newblock {\em Mathematische Annalen}, 276:225--254, 1986.

\bibitem{OptimWol}
D.~E. Campo-Duarte, O.~Vasilieva, D.~Cardona-Salgado, and M.~Svinin.
\newblock Optimal control approach for establishing {wMelpop} {Wolbachia} infection among wild {Aedes} aegypti populations.
\newblock {\em J. Math. Biol.}, 76(7):1907--1950, 2018.

\bibitem{CHIA2}
F.~Chiacchio.
\newblock Steiner symmetrization for an elliptic problem with lower-order terms.
\newblock {\em Ricerche Mat.}, 53(1):87--106, 2004.

\bibitem{CHIA}
F.~Chiacchio and V.~M. Monetti.
\newblock Comparison results for solutions of elliptic problems via {S}teiner symmetrization.
\newblock {\em Differential Integral Equations}, 14(11):1351--1366, 2001.

\bibitem{Chiti}
G.~Chiti.
\newblock Orlicz norms of the solutions of a class of elliptic equations.
\newblock {\em Boll. Un. Mat. Ital. A (5)}, 16(1):178--185, 1979.

\bibitem{CUCCU}
F.~Cuccu and G.~Porru.
\newblock Symmetry of solutions to optimization problems related to partial differential equations.
\newblock {\em Proceedings of the Royal Society of Edinburgh: Section A Mathematics}, 136(5):921–934, 2006.

\bibitem{DIAGOM}
J.~I. D\'iaz and D.~G\'omez-Castro.
\newblock Steiner symmetrization for concave semilinear elliptic and parabolic equations and the obstacle problem.
\newblock {\em Discrete Contin. Dyn. Syst.}, pages 379--386, 2015.

\bibitem{WOLB}
{Duprez, Michel}, {Hélie, Romane}, {Privat, Yannick}, and {Vauchelet, Nicolas}.
\newblock Optimization of spatial control strategies for population replacement, application to wolbachia.
\newblock {\em ESAIM: COCV}, 27:74, 2021.

\bibitem{BEH}
B.~Emamizadeh and Y.~Liu.
\newblock Bang-bang and multiple valued optimal solutions of control problems related to quasi-linear elliptic equations.
\newblock {\em SIAM Journal on Control and Optimization}, 58(2):1103--1117, 2020.

\bibitem{FERMER}
V.~Ferone and A.~Mercaldo.
\newblock Neumann problems and {S}teiner symmetrization.
\newblock {\em Comm. Partial Differential Equations}, 30(10-12):1537--1553, 2005.

\bibitem{FERO}
V.~Ferone and M.~R. Posteraro.
\newblock {Maximization on classes of functions with fixed rearrangement}.
\newblock {\em Differential and Integral Equations}, 4(4):707 -- 718, 1991.

\bibitem{TRU}
D.~Gilbarg and N.~S. Trudinger.
\newblock {\em Elliptic Partial Differential Equations of Second Order}, volume 224 of {\em Classics in Mathematics}.
\newblock Springer, 2001.

\bibitem{GIR}
L.~Girardin and B.~Maucourt.
\newblock Agro-ecological control of a pest-host system: Preventing spreading.
\newblock {\em SIAM Journal on Applied Mathematics}, 83(3):1172--1195, 2023.

\bibitem{GREENHALGH}
D.~Greenhalgh.
\newblock Some results on optimal control applied to epidemics.
\newblock {\em Mathematical Biosciences}, 88(2):125--158, 1988.

\bibitem{HECHT}
F.~Hecht.
\newblock New development in freefem++.
\newblock {\em J. Numer. Math.}, 20(3-4):251--265, 2012.

\bibitem{Henrot}
A.~Henrot.
\newblock {\em Extremum problems for eigenvalues of elliptic operators}.
\newblock Frontiers in Mathematics. Birkh\"{a}user Verlag, Basel, 2006.

\bibitem{KAW}
B.~Kawohl.
\newblock {\em Rearrangements and convexity of level sets in {PDE}}, volume 1150 of {\em Lecture Notes in Mathematics}.
\newblock Springer-Verlag, Berlin, 1985.

\bibitem{KESAVAN}
S.~Kesavan.
\newblock {\em Symmetrization \& applications}, volume~3 of {\em Series in Analysis}.
\newblock World Scientific Publishing Co. Pte. Ltd., Hackensack, NJ, 2006.

\bibitem{LANG}
J.~J. Langford.
\newblock Symmetrization of {P}oisson's equation with {N}eumann boundary conditions.
\newblock {\em Ann. Sc. Norm. Super. Pisa Cl. Sci. (5)}, 14(4):1025--1063, 2015.

\bibitem{FENG}
F.~Li and W.~Li.
\newblock Comparison results for solutions of elliptic {N}eumann problems with lower-order terms via {S}teiner symmetrization.
\newblock {\em Georgian Math. J.}, 26(1):83--96, 2019.

\bibitem{PLL}
P.-L. Lions.
\newblock {\em «Quelques remarques sur la symétrisation de Schwartz». Nonlinear Partial Differential Equations and Their Applications: Coll{\`e}ge de France, Seminar}, volume 1, 309--319.
\newblock Pitman, London, 1980.

\bibitem{LOU2}
Y.~Lou.
\newblock On the effects of migration and spatial heterogeneity on single and multiple species.
\newblock {\em J. Differential Equations}, 223(2):400--426, 2006.

\bibitem{YANA}
Y.~Lou and E.~Yanagida.
\newblock Minimization of the principal eigenvalue for an elliptic boundary value problem with indefinite weight, and applications to population dynamics.
\newblock {\em Japan J. Indust. Appl. Math.}, 23(3):275--292, 2006.

\bibitem{MAZPRINAD}
I.~Mazari, G.~Nadin, and Y.~Privat.
\newblock Optimisation of the total population size for logistic diffusive equations: bang-bang property and fragmentation rate.
\newblock {\em Communications in Partial Differential Equations}, 47(4):797--828, 2022.

\bibitem{MAZ}
I.~Mazari-Fouquer.
\newblock Optimising the carrying capacity in logistic diffusive models: Some qualitative results.
\newblock {\em Journal of Differential Equations}, 393:238--277, 2024.

\bibitem{MOSS}
J.~Mossino and J.-M. Rakotoson.
\newblock Isoperimetric inequalities in parabolic equations.
\newblock {\em Ann. Scuola Norm. Sup. Pisa Cl. Sci. (4)}, 13(1):51--73, 1986.

\bibitem{LOU}
K.~Nagahara, Y.~Lou, and E.~Yanagida.
\newblock Maximizing the total population with logistic growth in a patchy environment.
\newblock {\em J. Math. Biol.}, 82(1-2):Paper No. 2, 50, 2021.

\bibitem{EIJI}
K.~Nagahara and E.~Yanagida.
\newblock Maximization of the total population in a reaction-diffusion model with logistic growth.
\newblock {\em Calc. Var. Partial Differential Equations}, 57(3):Paper No. 80, 14, 2018.

\bibitem{OptimCOVID}
Q.~Richard, S.~Alizon, M.~Choisy, M.~T. Sofonea, and R.~Djidjou-Demasse.
\newblock Age-structured non-pharmaceutical interventions for optimal control of covid-19 epidemic.
\newblock {\em PLOS Computational Biology}, 17(3):1--25, 03 2021.

\bibitem{OptimInsecticide}
L.~S. Sepulveda-Salcedo, O.~Vasilieva, and M.~Svinin.
\newblock Optimal control of dengue epidemic outbreaks under limited resources.
\newblock {\em Stud. Appl. Math.}, 144(2):185--212, 2020.

\bibitem{TAL}
G.~Talenti.
\newblock Elliptic equations and rearrangements.
\newblock {\em Ann. Scuola Norm. Sup. Pisa Cl. Sci. (4)}, 3(4):697--718, 1976.

\bibitem{Trombetti}
G.~Trombetti.
\newblock Metodi di simmetrizzazione nelle equazioni alle derivate parziali.
\newblock {\em Bollettino dell'Unione Matematica Italiana}, 3-B(3):601--634, 10 2000.

\bibitem{TROMBV}
G.~Trombetti and J.~L. V\'{a}zquez.
\newblock A symmetrization result for elliptic equations with lower-order terms.
\newblock {\em Ann. Fac. Sci. Toulouse Math. (5)}, 7(2):137--150, 1985.

\bibitem{VAZ}
J.~L. V\'azquez.
\newblock Sym\'etrisation pour {$u\sb{t}=\Delta \varphi (u)$} et applications.
\newblock {\em C. R. Acad. Sci. Paris S\'er. I Math.}, 295(2):71--74, 1982.

\bibitem{VOAS}
C.~Voas and D.~Yaniro.
\newblock Symmetrization and optimal control for elliptic equations.
\newblock {\em Proceedings of the American Mathematical Society}, 99(3):509--514, 1987.

\end{thebibliography}

\end{document}